\documentclass[11pt,letterpaper]{amsart}
\usepackage{amsmath,amssymb,amsthm,graphicx,mathrsfs,url}
\usepackage[usenames,dvipsnames]{color}
\usepackage[colorlinks=true,linkcolor=Red,citecolor=Green]{hyperref}
\usepackage{amsxtra}
\usepackage{wasysym} % for "\varhexagon" macro
\usepackage{graphicx}% for "\scalebox" macro
\usepackage{subcaption}
\usepackage{array}
\newcolumntype{P}[1]{>{\centering\arraybackslash}m{#1}}
\usepackage{biblatex}
\usepackage[english]{babel}
\usepackage[autostyle, english = american]{csquotes}
\MakeOuterQuote{"}

\usepackage{blindtext}

\newtheorem{theo}{Theorem}
\newtheorem{prop}{Proposition}[section]	
\newtheorem{defi}[prop]{Definition}
\newtheorem{conj}{Conjecture}

\newtheorem{lemm}[prop]{Lemma}
{}
\newtheorem{rem}{Remark}
\numberwithin{equation}{section}

\newtheorem{rmk}{Remark}

\newtheorem{manualtheoreminner}{Corollary}
\newenvironment{manualtheorem}[1]{%
  \renewcommand\themanualtheoreminner{#1}%
  \manualtheoreminner
}{\endmanualtheoreminner}

\DeclareMathOperator{\tr}{tr}

\usepackage{scalerel}

\newcommand\reallywidehat[1]{\arraycolsep=0pt\relax%
\begin{array}{c}
\stretchto{
  \scaleto{
    \scalerel*[\widthof{\ensuremath{#1}}]{\kern-.5pt\bigwedge\kern-.5pt}
    {\rule[-\textheight/2]{1ex}{\textheight}} %WIDTH-LIMITED BIG WEDGE
  }{\textheight} % 
}{0.5ex}\\           % THIS SQUEEZES THE WEDGE TO 0.5ex HEIGHT
#1\\                 % THIS STACKS THE WEDGE ATOP THE ARGUMENT
\rule{-1ex}{0ex}
\end{array}
}
\author{Tristan Humbert}
\email{humbertt@imj-prg.fr}
\address{Sorbonne Université, Paris France 75005.}
\begin{document}
\begin{abstract}
We show that, given a real or complex hyperbolic metric $g_0$ on a closed manifold $M$ of dimension $n\geq 3$, there exists a neighborhood $\mathcal U$ of $g_0$ in the space of negatively curved metrics such that for any $g\in \mathcal U$, the topological entropy and Liouville entropy of $g$ coincide if and only if $g$ and $g_0$ are homothetic. This provides a partial answer to Katok's entropy rigidity conjecture. As a direct consequence of our theorem, we obtain a local rigidity result of the hyperbolic rank near complex hyperbolic metrics.
\end{abstract}
\title{Katok's entropy conjecture near real and complex hyperbolic metrics} 
\maketitle
\section{Introduction}
\subsection{Statement of the problem}
Let $(M^n,g)$ be a smooth closed $n$-dimensional manifold equipped with a smooth Riemannian metric $g$ of negative sectional curvature. The geodesic flow $\varphi_t$ on the unit tangent bundle $SM:=\{(x,v)\in TM\mid \|v\|_g=1\}$ is an \emph{Anosov flow}, which means that if we denote by $X=\tfrac{d}{dt}\varphi_t|_{t=0}$ the geodesic vector field, then there exists a flow-invariant, continuous splitting $T(SM)=E_u\oplus \mathbb R X\oplus E_s$ of the tangent bundle and uniform constants $C,\theta>0$ such that for any $p\in SM$, 
\begin{align*}
&\|d\varphi_t (p)v_s\|\leq Ce^{-\theta t}\|v_s\|, \ v_s\in E_s(x), \  t\geq 0,
\\&\|d\varphi_t (p)v_u\|\leq Ce^{-\theta |t|}\|v_u\|, \ v_u\in E_u(x), \ t\leq 0.
\end{align*}
The metric $g$ lifts to a metric $g_{\mathrm{Sas}}$ on $SM$ called the Sasaki metric (see \cite[Chapter 1]{Pat}) and the norms in the previous inequalities are taken with respect to $g_{\mathrm{Sas}}$. The Sasaki metric defines a natural smooth Riemannian volume form $\mu_{\mathrm{Liou}}$ which is also the Liouville form associated to the contact structure of $SM$. The (normalized) measure $\mu_{\mathrm{Liou}}$ is called the Liouville measure, it is smooth and invariant by the flow.

An Anosov flow has infinitely many invariant probability measures. The following result is classical and known as the \emph{variational principle}, see for instance \cite[Corollary 4.3.9]{FishHas},
\begin{equation}
\label{eq:varia}
\mathrm{Ent}_{\mathrm{top}}(g):=\mathrm{Ent}_{\mathrm{top}}(\varphi_1^g)=\sup\{ \mathrm{Ent}(\varphi_1^g,\mu)\mid \mu \text{ invariant  probability measure}\}.
\end{equation}
Here, $\mathrm{Ent}_{\mathrm{top}}(\varphi_1^g)$ denotes the topological entropy of the time-one map and $\mathrm{Ent}(\varphi_1^g,\mu)$ denotes the metric entropy of the time one map with respect to the measure $\mu$.
Moreover, for an Anosov flow, the supremum is attained for a unique invariant measure $\mu_{\mathrm{BM}}$ for which we have $\mathrm{Ent}(\varphi_1^g,\mu_{\mathrm{BM}})=\mathrm{Ent}_{\mathrm{top}}(g)$.
This measure $\mu_{\mathrm{BM}}$ is the measure of maximal entropy (or Bowen-Margulis measure). A natural question, first raised by Katok in \cite[Section 2]{Ka}, is to characterize the negatively curved metrics $g$ for which the Liouville measure and the measure of maximal entropy coincide. Note that because of the variational principle recalled above, this is equivalent to $\mathrm{Ent}_{\mathrm{Liou}}(g):=\mathrm{Ent}(\varphi_1^g,\mu_{\mathrm{Liou}})=\mathrm{Ent}_{\mathrm{top}}(g)$.
The following conjecture is known as \emph{Katok's entropy rigidity conjecture} \cite{Ka,BK}.
\begin{conj}[Katok]
\label{conjKatok}
For a closed negatively curved manifold $(M^n,g)$, one has $\mathrm{Ent}_{\mathrm{Liou}}(g)=\mathrm{Ent}_{\mathrm{top}}(g)$ if and only if $g$ is locally symmetric. 
\end{conj}
A negatively curved locally symmetric manifold is covered by a either a real,
complex, quarternionic or octonionic hyperbolic space, see Subsection \ref{locsym} for more details.

The main theorem of the current paper is the following. Recall that two metrics $g_1$ and $g_2$ are homothetic if there exists a constant $\lambda>0$ such that $g_1$ and $\lambda g_2$ are isometric.
\begin{theo}[Entropy rigidity near real and complex hyperbolic spaces]
\label{newtheo}
Let $(M^n,g_0)$ be a closed manifold of dimension $n\geq 3$ and suppose that $g_0$ is either a real hyperbolic or a complex hyperbolic metric. Then there exists $N(n)\in \mathbb N$ and $\epsilon>0$ such that for any negatively curved metric $g$ on $M$ with $\|g-g_0\|_{C^N}<\epsilon$, if one has $\mathrm{Ent}_{\mathrm{Liou}}(g)=\mathrm{Ent}_{\mathrm{top}}(g),$ then $g$ is homothetic to $g_0$ and thus locally symmetric.
\end{theo}
We actually prove a more precise result which holds for any type of locally symmetric metric (which we will make explicit in the rest of the introduction), see Theorem \ref{MainTheo} below.
\subsection{Existing results}
Katok showed  Conjecture \ref{conjKatok} for surfaces (i.e., $n=2$) \cite{Ka}, and his proof relied on the fact that any negatively curved metric is conformally equivalent to a hyperbolic metric in this case. His argument extends to any dimension for metrics in the conformal class of a locally symmetric metric but this does not exhaust all metrics in dimension $n\geq 3$.

We note that Foulon \cite{Fou}, and then De Simoi, Leguil, Vinhage and Yang \cite{SLVY}, generalized the result of Katok. In their setting, the geodesic flow on the 3-dimensional manifold $SM$ is replaced by a general Anosov flow on a 3-manifold. The Liouville measure is replaced by the contact volume or a smooth invariant measure, and they show that equality with the measure of maximal entropy implies that the system is conjugate to an algebraic flow (see \cite{Fou,SLVY} for the precise statements).

In higher dimensions, fewer results are known. Let us first state a major result of Besson, Courtois, Gallot \cite{BCG} which asserts that for a closed manifold $M^n$ which admits a locally symmetric metric $g_0$ (unique in this case by Mostow's rigidity theorem), then $g\mapsto\mathrm{Ent}_{\mathrm{top}}(g)$ is (strictly) minimized at $g_0$ among negatively curved metrics $g$ with volume equal to that of $g_0$. For the entropy of the Liouville measure, Flaminio \cite[Theorem C]{Fla} showed that there exist hyperbolic metrics in dimension $3$ which neither maximize nor minimize $g\mapsto \mathrm{Ent}_{\mathrm{Liou}}(g)$ among metrics of constant volume, see also \cite{Mau} for a generalization to higher dimensions.

Flaminio in \cite[Theorem A]{Fla} obtained an "infinitesimal" version of Katok's conjecture near real hyperbolic metrics in any dimension. To state his result, let us first introduce a functional $\Phi$ defined on negatively curved metrics:
\begin{equation}
\label{eq:phi}
 \Phi(g):=\mathrm{Ent}_{\mathrm{top}}(g)-\mathrm{Ent}_{\mathrm{Liou}}(g)\geq 0.
\end{equation}
By the variational principle \eqref{eq:varia}, the equality of $\mu_{\mathrm{Liou}}^g=\mu_{\mathrm{BM}}^g$ is equivalent to $\Phi(g)=0$. We note that there exists a natural gauge by the action of the group $\mathrm{Diff}_0(M)$ of diffeomorphisms isotopic to the identity. We define
$$\mathcal O(g_0):=\{\phi^*g_0\mid \phi \in \mathrm{Diff}_0(M)\},\ T_{g_0}\mathcal O(g_0):=\{\mathcal L_Vg_0\mid V\in C^{\infty}(M;TM)\}. $$
In particular, the functional $\Phi$ is constant along any orbit $\mathcal O(g),$ and the kernel of the Hessian $d^2\Phi(g_0)$ always contains the tangent space $T_{g_0}\mathcal O(g_0)$. If $g_0$ is locally-symmetric, then $\Phi(g_0)=0$ and $d\Phi(g_0)=0$ as $g_0$ is a critical point of $\Phi$. Then Flaminio proved that for a real hyperbolic metric in any dimension $n\geq 3$, the Hessian $d^2\Phi(g_0)$ is a positive quadratic form with kernel given exactly by $T_{g_0}\mathcal O(g_0)$. This infinitesimal version of Katok's entropy conjecture shows that if we consider a smooth deformation $(g_\lambda)_{\lambda \in (-\epsilon,\epsilon)}$ of $g_0$ (in the space of negatively curved metrics of same volume) for which one has $\Phi(g_\lambda)\equiv 0$, then the deformation is tangent to the orbit of $g_0$ at $g_0$.

\subsection{Entropy rigidity near locally symmetric metrics}
In this paper, the analysis is made "transversally" to the tangent space $T_{g_0}\mathcal O(g_0)$. For a nearby metric $g$, we will first need to "project" it on $\mathrm{Ker}(D_{g_0}^*):=(T_{g_0}\mathcal O(g_0))^{\perp}$ (the space of divergence-free symmetric two-tensors, see Lemma \ref{lemmD*}). For $g$ close enough to $g_0$, its orbit $\mathcal O(g)$ intersects $\mathrm{Ker}(D_{g_0}^*)$ at one single point which we will denote by $\phi_g^*g$, see Lemma \ref{slice lemma} for a precise statement. We will denote by $\Pi_{\mathrm{Ker}(D_{g_0}^*)}$ the orthogonal projection from the space of symmetric tensors onto $\mathrm{Ker}(D_{g_0}^*)$.
Using microlocal methods, we prove the following result, which is valid for any locally symmetric metric. 
\begin{theo}
\label{MainTheo}
Let $(M^n,g_0)$ be a closed manifold with $n\geq 3$ and $g_0$ a locally symmetric negatively curved metric. Then there exists a differential operator $Q$ $($made explicit in Theorem \ref{theoHess}$)$   and an open neighborhood $\mathcal U$ of $g_0$ in the $C^{ 3n/2 +7}$-topology, such that if $g\in \mathcal U$ satisfies $\phi_g^*g-g_0\in \mathrm{Ker}(\Pi_{\mathrm{Ker}(D_{g_0}^*)} Q\Pi_{\mathrm{Ker}(D_{g_0}^*)})^{\perp}$, one has
\begin{equation}
\label{eq:stability}
\|\phi_g^*g-g_0\|_{H^{3/2}}^2\leq C_n(\mathrm{Ent}_{\mathrm{top}}(g)-\mathrm{Ent}_{\mathrm{Liou}}(g))+C_n(\mathrm{Vol}_g(M)-\mathrm{Vol}_{g_0}(M))^2
\end{equation}
for some constant $C_n>0$. In particular, if $\mathrm{Vol}_g(M)=\mathrm{Vol}_{g_0}(M)$ and $\mathrm{Ent}_{\mathrm{top}}(g)=\mathrm{Ent}_{\mathrm{Liou}}(g)$, then $g$ is isometric to $g_0$ and thus locally symmetric.
\end{theo}
The meaning of the previous theorem is the following. For any locally symmetric metric $g_0$, there is a finite-dimensional space of directions (given by the kernel of $\Pi_{\mathrm{Ker}(D_{g_0}^*)} Q\Pi_{\mathrm{Ker}(D_{g_0}^*)}$) in the transverse slice $\mathrm{Ker}(D_{g_0}^*)$ such that if a nearby metric $g$ projects orthogonally to it, the equalities of entropies and volumes imply that $g$ is locally symmetric. Here, $H^{3/2}$ denotes the usual Sobolev space. Note that using the boundedness in $C^{3n/2+7}$-norm and an interpolation argument, one could replace the $H^{3/2}$-norm in the left hand side by a $C^{k}$-norm for any $k<3n/2$, at the cost of replacing the right hand side by some (explicit) power $\delta(k,n)>0$ of it.

The proof of Theorem \ref{newtheo} then reduces to proving that the aforementioned operator $\Pi_{\mathrm{Ker}(D_{g_0}^*)} Q\Pi_{\mathrm{Ker}(D_{g_0}^*)}$ is injective. With our current techniques, we can only prove this for real and complex hyperbolic metrics.
\begin{theo}[Solenoidal injectivity]
\label{soltheo}
The operator $\Pi_{\mathrm{Ker}(D_{g_0}^*)} Q\Pi_{\mathrm{Ker}(D_{g_0}^*)}$ is injective if $g_0$ is a real or complex hyperbolic metric.
\end{theo}
We conjecture that the operator $\Pi_{\mathrm{Ker}(D_{g_0}^*)} Q\Pi_{\mathrm{Ker}(D_{g_0}^*)}$ is also injective for quaternionic and octonionic hyperbolic metrics but are currently unable to prove it, see Conjecture \ref{conj}. This would give a local entropy rigidity result near any locally symmetric metric.
 We insist on the fact that this step seems to be purely geometrical and not to involve microlocal analysis.
\subsection{Local hyperbolic rank rigidity}
A Riemannian manifold $(M,g)$ has \emph{higher hyperbolic rank} if for any geodesic $\gamma$, there is a non-vanishing Jacobi field $J(t)$ along $\gamma$ such that $J(t)$ and $\dot \gamma(t)$ span a plane of curvature equal to $-1$. Higher hyperbolic rank on closed manifold is conjectured to only occur for locally symmetric spaces. We list thereafter some known results  (see \cite{CNS20} for a more detailed account of the existing literature):
\begin{itemize}
\item if  $(M,g)$ has \emph{higher hyperbolic rank} and its sectional curvature satisfies $K\leq -1$ then it is a locally symmetric space of rank $1$ \cite{Ham},
\item if  $(M,g)$ has \emph{higher hyperbolic rank} and its sectional curvature is 1/4-pinched (that is $-1\leq K\leq -1/4$) then it is a locally symmetric space of rank $1$ \cite[Theorem 1.1]{CNS20}. See also \cite[Corollary 1]{Cons} for similar results under different pinching conditions.
\end{itemize}
In \cite{CNS20}, the authors made the following conjecture.
\begin{conj}[Connell-Nguyen-Spatzier]
A closed negatively curved manifold $(M^n,g)$ with higher hyperbolic rank and curvature satisfying $K\geq -1$ is locally symmetric.
\end{conj}
In a subsequent paper \cite{CNS21}, they prove the following local results near locally symmetric metrics.
\begin{itemize}
\item Let  $(M,g_0)$ be either a closed quaternionic or octonionic hyperbolic locally symmetric manifold. Then there exists a neighborhood $\mathcal U$ of $g_0$ in the $C^3$-topology such that if $g\in \mathcal U$ has \emph{higher hyperbolic rank} with sectional curvature satisfying $K\geq -1$, then it is a locally symmetric space \cite[Theorem 1.14]{CNS21}.
\item Let  $(M,g_0)$ be a closed complex  hyperbolic locally symmetric manifold. Then there exists a neighborhood $\mathcal U$ of $g_0$ in the $C^3$-topology such that if $g\in \mathcal U$ has \emph{higher hyperbolic rank} with sectional curvature satisfying $K\geq -1$, then its Liouville entropy and topological entropy coincide \cite[Theorem 1.15]{CNS21}.  \end{itemize}
Their techniques do not allow them to prove the local rigidity of hyperbolic rank near complex hyperbolic metrics. However, combining their result with Theorem \ref{newtheo}, we obtain:

\begin{manualtheorem}{1}[Local rigidity of hyperbolic rank near complex hyperbolic metrics]
Let  $(M^n,g_0)$ be a closed complex  hyperbolic locally symmetric manifold. Then there exists a neighborhood $\mathcal U$ of $g_0$ in the $C^{3n/2+7}$-topology such that if $g\in \mathcal U$ has \emph{higher hyperbolic rank} with sectional curvature satisfying $K\geq -1$, then $g$ is locally symmetric.
\end{manualtheorem}
This is an analogue of  \cite[Theorem 1.14]{CNS21} for complex hyperbolic metrics. Note however that with our techniques, the neighborhood is taken in the $C^{3n/2+7}$-topology while their results hold for $C^3$ metrics.
\subsection{Local rigidity for metric with $C^2$ stable and unstable foliations.}
The regularity of the Anosov splitting associated to the geodesic flow on a negatively curved manifold is a property that characterizes locally-symmetric metrics along all negatively curved metrics. More precisely:
\begin{itemize}
    \item Hurder and Katok \cite{HuKa} showed that for negatively curved surfaces, if the Anosov splitting is $C^2$, then it is $C^{\infty}$.
    \item Kanai \cite{Kan} showed that in dimension $n\geq 3$ and for $4/9$-pinched negatively curved metrics, smoothness of the Anosov splitting implied that $g$ is isometric to a real hyperbolic metric. See also Feres and Katok \cite{FeKa1,FeKa2} for the same result with no pinching condition in dimension $3$ and the optimal pinching in dimension $4$.
    \item Benoist, Foulon and Labourie \cite{BFL} showed that in any dimension $n$, there exists $k(n)\in \mathbb N$ such that if a negatively curved metric has $C^k$ foliations, then it is locally symmetric. 
\end{itemize}
The value $k$ in the previous theorem grows with the dimension $n$ and one can take $k(n)=2(2n^2-n+1)$. However, as noticed by Hamenstädt in \cite[Theorem B]{Ham2}, for a negatively curved metric, if the Anosov splitting is $C^2$, then the Liouville measure is equal to the measure of maximal entropy. From Theorem \ref{newtheo}, we thus obtain:
\begin{theo}[Local rigidity of metric with $C^2$ foliation near real and complex hyperbolic metrics]
Let  $(M^n,g_0)$ be a closed real or complex  hyperbolic locally symmetric manifold. Then there exists a neighborhood $\mathcal U$ of $g_0$ in the $C^{3n/2+7}$-topology such that if $g\in \mathcal U$ has $C^2$ Anosov foliation, then $g$ is locally symmetric.
    
\end{theo}

\subsection{Strategy of the proof.}
The first part of Flaminio's argument consists in using the geometry of real hyperbolic spaces to first simplify the expression of the Hessian $d^2\Phi(g_0)$. We start by generalizing the computation to all locally symmetric metrics (see Theorem \ref{eq:theoHess}). The computation relies on the geometrical characterization of locally symmetric spaces recalled in Subsection \ref{locsym}.

The main difference with Flaminio's strategy is that we rewrite the Hessian using the "generalized X-ray transform" $\Pi$. The basic properties of the operator $\Pi$ are recalled in Subsection \ref{Xray}.

Note that the ellipticity and injectivity of $\Pi$ on suitable spaces of symmetric tensors was the core of the arguments of \cite{GuLef,GuKnLef}.
Our proof share similarities with the ones of \cite{GuLef,GuKnLef}. In particular, the stability estimate \eqref{eq:stability} is deduced from a Taylor expansion once one shows the Hessian $d^2\Phi(g_0)$ is elliptic and injective when restricted to solenoidal tensors, which we refer to as the \emph{infinitesimal problem.} The ellipticity of the Hessian is deduced from the works of Guillarmou and Guillarmou-Lefeuvre \cite[Theorem 3.5, Lemma 4.3]{Gu,GuLef}) and an explicit computation of the principal symbol of the differential operator $Q$ appearing in Theorem \ref{eq:theoHess}, see Proposition \ref{elliptic}.

The infinitesimal problem for real hyperbolic metrics was solved by Flaminio using a Weitzenböck formula. In Section \ref{sinj}, we show Theorem \ref{soltheo}, that is, we solve the infinitesimal problem for real and complex hyperbolic metrics.

The proof of Theorem \ref{soltheo} for complex hyperbolic metrics is much more involved than for real hyperbolic metrics. The technical reason is that the differential operator $Q$ appearing in the Hessian computation takes values in symmetric $4$-tensors for the complex hyperbolic case while it takes values in symmetric $2$-tensors for the real hyperbolic one. Note that the infinitesimal problem corresponding to the marked length spectrum rigidity conjecture studied in \cite{GuLef,GuKnLef} was solved by Croke and Sharafutdinov \cite{CrSh} using ideas from tensor tomography and inverse problems. We will follow this approach to prove Theorem \ref{soltheo}.

More precisely, we will use a correspondance between symmetric tensors and spherical harmonics recalled in Subsection \ref{SM}. The complex structure allows to further decompose the spaces of spherical harmonics. We develop in Subsection \ref{CH} a theory of raising/lowering operators for general complex hyperbolic spaces very similar to the work of Guillemin–Kazhdan on surfaces \cite{GK} which might be of independent interest. In Section \ref{sinj}, we use this theory as well as a Weitzenböck formula to complete the proof of Theorem \ref{soltheo}.

\textbf{Acknowledgements.} The author would like to first thank Colin Guillarmou and Thibault Lefeuvre for their careful and precious guidance during the writing of this paper.

The author would also like to thank Livio Flaminio and François Ledrappier for interesting discussions on Katok's conjecture. He would like to further thank Livio Flaminio for pointing out that our main theorem could be applied to solve the local rigidity of higher hyperbolic rank near complex hyperbolic metrics.

The author would like to thank Andrei Moroianu for answering some questions on complex and quaternionic geometry.

Finally, the author would like to thank the anonymous referee for their careful and thorough reading of the paper.

This project has received funding from the European Research Council (ERC) under the European Union’s Horizon research and innovation programme (grant agreement No. 101162990).
\section{Preliminaries}
We recall several results we will need to prove Theorem \ref{MainTheo}.
\label{preli}
\subsection{Symmetric tensors and analysis on $SM$.}
\label{SM}
In this subsection, we recall some notions and tools from tensor tomography. We refer to \cite[Chapter 5.3]{GuMaz} for a textbook account of these notions.
\subsubsection{Tensorial analysis}
We fix a Euclidean space $V$ and  identify $V$ with its dual $V^*$ using its inner product. Let $m$ be an integer. For an $m$-tensor $f\in V^{\otimes m}$, we say that $f$ is a symmetric if for any permutation  $\sigma\in \mathfrak S_m$,
$$ \forall (v_1,\ldots,v_m)\in V^m,\,\, f(v_{\sigma(1)},\ldots,v_{\sigma(m)})=f(v_1,\ldots, v_m). $$
The set of symmetric $m$-tensors will be denoted by $S^mV$. The \emph{symmetrisation} of an $m$-tensor is defined to be its (orthogonal) projection on $S^mV$, it is defined by 
\begin{equation}
\label{eq:S}
\mathrm{Sym}: V^{\otimes m}\to S^mV,\quad e_{i_1}\otimes e_{i_2}\otimes \cdots \otimes e_{i_m}\mapsto \frac{1}{m!}\sum_{ \sigma\in \mathfrak S_m}e_{i_{\sigma(1)}}\otimes e_{i_{\sigma(2)}}\otimes \cdots \otimes e_{i_{\sigma(m)}}
\end{equation}
where $(e_1,\ldots ,e_n)$ denotes a basis of $V$ and the resulting tensor is extended by linearity.

An important operator is given by the \emph{trace}:
\begin{equation}
\label{eq:trace}
\tr: S^mV\to S^{m-2}V,\quad f\mapsto \sum_{i=1}^mf(e_i,e_i,\cdot,\ldots,\cdot), 
\end{equation}
where $(e_i)_{i=1}^n$ is an orthonormal basis of $V$. The scalar product on $V$ extends naturally to the space of symmetric tensors:
$$\forall g,h \in S^m V, \ \langle h,g\rangle_{S^mV}:=\frac{1}{m!}\sum_{i_1,\ldots, i_m}h(e_{i_1},\ldots,e_{i_m})g(e_{i_1},\ldots,e_{i_m}). $$
The adjoint $L$ of the trace  for this scalar product  is then is given by:
\begin{equation}
\label{eq:adjointt}
L: S^mV\to S^{m+2}V,\ L(h):=\mathrm{Sym}(h\otimes g_0),
\end{equation}
where $g_0$ is the symmetric $2$-tensor given by the metric : $g_0=\sum_{i=1}^n e_i\otimes e_i$ for any orthonormal basis $(e_i)_{i=1}^n$. The operator $L$ is easily seen to be injective. In particular, if we denote by $S_0^mV=\mathrm{Ker}(\tr)\cap S^mV$ the trace-free symmetric $m$-tensors, then
\begin{equation}
\label{eq:tensordecomp}
S^mV=\bigoplus_{k=0}^{\lfloor m/2\rfloor}L^k\big(S_0^{m-2k}V\big).
\end{equation}
We now specialize to the case $V=T_xM$, more precisely, an $m$-tensor on $M$ will denote a smooth section of the bundle $C^{\infty}(M;S^mT^*M)$. The previous operators extend naturally and so does the decomposition \eqref{eq:tensordecomp}.
\subsubsection{Analysis on $SM$}
There is a corresponding decomposition to \eqref{eq:tensordecomp} in terms of {spherical harmonics}. This will be sometimes more convenient for the computations. 

We consider the unit tangent bundle $SM:=\{(x,v)\in TM\mid g(v,v)=1\}$. We will write $\varphi_t$ for the geodesic flow. The geodesic vector field is then the generator of the flow, i.e., $X=\tfrac{d}{dt}\varphi_t|_{t=0}.$
We denote by $\pi:SM\to M$ the canonical projection on the base. The vertical distribution is defined to be 
$$
 \mathbb V:=\mathrm{Ker}(d\pi).
$$
The horizontal distribution is defined using parallel transport
$$
{\mathbb H}(x,v):=\left\{\frac{d}{dt}\alpha(t)|_{t=0}\mid \forall W\in T_{x,v}SM,\alpha(t)=\big(\pi(\varphi_t(x,v)),\mathcal P_t W\big),  g_{\text{Sas}}(X,W)=0\right\}
$$
where $\mathcal P_t$ denotes the parallel transport along the geodesic. We recall that the metric $g$ on the base induces naturally a metric $g_{\text{Sas}}$ (called the Sasaki metric, see \cite[Chapter 1]{Pat}) on the unit tangent bundle $SM$. We have the following orthogonal splitting: 
$$
T(SM)=\mathbb R X\oplus {\mathbb H}\oplus \mathbb V.
$$

The Levi-Civita connection decomposes as $\nabla_{\text{Sas}}f=(Xf)X+\widetilde{\nabla_{\mathbb H}}f+\widetilde{\nabla_{\mathbb V}}f$, where $\widetilde{\nabla_{\mathbb H}}$ is the \emph{horizontal gradient} and $\widetilde{\nabla_{\mathbb V}}$ is the \emph{vertical gradient}. It is actually convenient to identify $\mathbb H$ and $\mathbb V$ with the same vector bundle $\mathcal N$ over $SM$ called the \emph{normal bundle} so that the two gradients act on the same space.  Define the normal bundle:
$$
\mathcal N(x,v):=\{w\in T_xM \mid g_{x}\big(w,d\pi(X(x,v))\big)=0\}.
$$
We then see that the horizontal bundle is isometric to the normal one and that
$$
d\pi_{(x,v)}: \mathbb H(x,v)\overset{\sim}{\longrightarrow} \mathcal N(x,v).
$$
Similarly, define the mapping $\mathcal K$ to be 
$$\mathcal K: \mathbb V \to \mathcal N,\quad W\mapsto \frac{d}{dt}|_{t=0}(\pi(\varphi_t(x,v)), \mathcal P_tW). $$
Then $\mathcal K$ is an isometry and we have
$$
\mathcal K(x,v): \mathbb V(x,v)\overset{\sim}{\longrightarrow} \mathcal N(x,v).
$$
We combine these two isometries to identify
$$
\mathbb H\oplus \mathbb V\overset{\sim}{\longrightarrow} \mathcal N\oplus \mathcal N,\quad (w,v)\mapsto (d\pi(w), \mathcal K(v)).
$$
We make the gradients act on the normal bundle using the isometries
$$
\nabla_{\mathbb H}:=d\pi \widetilde{\nabla_{\mathbb H}},\ \nabla_{\mathbb V}:=\mathcal K \widetilde{\nabla_{\mathbb V}}: \,  C^{\infty}(SM)\to C^{\infty}(SM,\mathcal N).
$$
The $L^2$-adjoint (where the $L^2$ space is taken with respect to the Liouville measure) for these maps are denoted by
$$
\nabla_{\mathbb H}^*,\nabla_{\mathbb V}^*:C^{\infty}(SM,\mathcal N)\to C^{\infty}(SM).
$$
The vertical Laplacian is then defined to be
$$
\Delta_{\mathbb V}:=\nabla_{\mathbb V}^*\nabla_{\mathbb V}: C^{\infty}(SM)\to C^{\infty}(SM).
$$
The vertical Laplacian coincides with the round Laplacian in each fiber of $SM$. 
$$\Delta_{\mathbb V}f(v)=\Delta_{S_xM}(f|_{S_xM}), \ \forall v\in T_xM. $$
We thus define the fiber bundle 
$$\Omega_m\longrightarrow M,\quad \Omega_m(x)=\mathrm{Ker}\big(\Delta_{S_xM}-m(n+m-2)\big), \ m\geq 0. $$
Each smooth function $f\in C^{\infty}(SM)$ decomposes uniquely into a sum of spherical harmonics
$$
f=\sum_{m=0}^{+\infty}f_m, \quad  f_m\in \Omega_m \iff \Delta_{\mathbb V}f_m=m(n+m-2)f.
$$
 The link with symmetric tensors is given by the pullback map: 
\begin{equation}
\label{eq:pi_k^*}
\pi_m^*:C^{\infty}(M,S^mTM)\to C^{\infty}(SM)
, \, f\mapsto \big( (x,v)\mapsto f(x)(v,\ldots,v)\big).
\end{equation}
It defines an isomorphism from the trace-free tensors of degree $m$ to the spherical harmonics of degree $m$:
\begin{equation}
\label{eq:iso}
\pi_m^*: C^{\infty}(M,S_0^mT^*M)\overset{\sim}{\longrightarrow} \Omega_m.
\end{equation}

The pull-back operator $\pi_m^*$ is conformal with respect to the natural scalar products of $C^{\infty}(M,S_0^mT^*M)$ and $\Omega_m$:
\begin{equation}
\label{eq:conformal}
\forall S,T\in C^{\infty}(M,S_0^mT^*M), \ \langle S,T\rangle_{L^2(S^2T^*M)}=\Lambda_m^n\langle \pi_m^*S,\pi_m^*T\rangle_{L^2(SM)},
\end{equation}
where we have 
$$\Lambda_m^n=\frac{\Gamma(n/2+m)}{2^{1-m}m!\pi^{n/2}}. $$
See \cite[Lemma 5.7]{GuMaz} for a proof. We see that the conformal factor $\Lambda_m^n$ depends on the degree $m$ and this will be a central point in the computations of Section \ref{sectionsinj}. In the following, we might suppress the index and write $\langle \cdot,\cdot\rangle$ for the scalar product if it is clear from the context which scalar product is used.

We will need a last correspondence between the spherical harmonics and symmetric tensors. One can define the total horizontal gradient $\nabla^{\mathrm{tot}}_{\mathbb H}$ by projecting the gradient on the total horizontal space $\mathbb H_{\mathrm{tot}}:=\mathbb H\oplus \mathbb R X$. One can then consider its adjoint $(\nabla^{\mathrm{tot}}_{\mathbb H})^*$ and the total horizontal Laplacian $\Delta_{\mathbb H}^{\mathrm{tot}}:=(\nabla^{\mathrm{tot}}_{\mathbb H})^*\nabla^{\mathrm{tot}}_{\mathbb H}$. For a function $f\in C^{\infty}(SM)$, we extend it to $C^{\infty}(TM\setminus\{0\})$ by homogeneity, we will still denote the extension by~$f$.
\begin{lemm}
\label{lemm deltaH}
One has the following relation
\begin{equation}
\label{eq:deltaH}
\forall S\in C^{\infty}(M;S^2T^*M), \ \pi_2^*(\nabla^*\nabla S)=\Delta_{\mathbb H}^{\mathrm{tot}}\pi_2^*S,
\end{equation}
where $\nabla^*$ is the $L^2$-adjoint of $\nabla$.
\end{lemm}
\begin{proof}
Consider a local frame $(e_1,e_2,\ldots, e_n)$ parallel at $x_0$ (i.e., $\nabla_{e_i}e_j=0$ at $x_0$).
Recall that the action of $\nabla^*$ on a $m$-tensor $T$ can be described locally by:
\begin{equation}
\label{eq:star}
\nabla^*T=-\tr(\nabla T)=-\sum_{i=1}^n\nabla_{e_i}T(e_i,\cdot, \ldots, \cdot). 
\end{equation}
In the following, we will use the following notation:
$$\forall T\in C^{\infty}(M,S^mT^*M), \ \forall X,Y\in C^{\infty}(M,TM), \ \nabla^2_{X,Y}T:=\nabla_X\nabla_Y T-\nabla_{\nabla_XY}T. $$
Then we compute
\begin{align*}
\pi_2^*(\nabla^* \nabla S)(x,v)&=-\sum_{j=1}^n \nabla^2_{e_j,e_j}S_x(v,v)
\\&=-\sum_{j=1}^ne_j(\nabla_{e_j}S_x(v, v))+2\sum_{j=1}^n\nabla_{e_j}S_x(\nabla_{e_j}v,v).
\end{align*}
Consider now $X_i(x,v):=d\pi_{(x,v)}^{-1}(e_i(x))$ the horizontal lifts of $e_i$. We first check that
\begin{align*}&\nabla_{e_j}v=\nabla_ve_j+[v,e_j]=\sum_{k=1}^n v_k \nabla_{e_k}e_j+[d\pi(x,v)X(x,v),d\pi(x,v)X_j(x,v)],
\\&\nabla_{e_j}v=d\pi(x,v)[X(x,v),X_j(x,v)]=d\pi(x,v)(R(e_i,v)v)^{\mathbb V}=0,
\end{align*}
where $(R(e_i,v)v)^{\mathbb V}\in \mathbb V$ is the vertical component of $R(e_i,v)v$ (in other words, one has $\mathcal K(x,v)(R(e_i,v)v)^{\mathbb V}=R(e_i,v)v$) and the computation of $[X,X_j]$ can be found in the proof of \cite[Lemma 13.1.7]{Lef}. This means that the expression above simplifies into
\begin{align*}
\pi_2^*(\nabla^* \nabla S)(x,v)=-\sum_{j=1}^ne_j(e_j(\pi_2^*S(x,v))).
\end{align*}
Next, consider a horizontal vector $W\in T_{(x_0,v_0)}SM$ and choose a path $(x(t),v(t))$ in $SM$ such that $(x(0),v(0))=(x_0,v_0)$ and $(\dot x(0),\dot v(0))=W$. We define
 $$S: M\to \mathrm{End}(T_xM),\ \ S_x(v,w)=g_x(S(x)v,w), \ \forall w,v\in T_xM.$$ 
 We now study the action of $W$ on $\pi_2^*S:$
 \begin{align*}
 W\pi_2^*S(x,v)=\frac{d}{dt}|_{t=0}g_x(S(x)v,v)=g_x\big(\nabla_{\dot x}(S(x)v)|_{t=0},v_0\big)+g_x(S(x_0)v_0,\nabla_{\dot x}v|_{t=0}).
 \end{align*}
First, the fact that $W$ is horizontal is equivalent to $\mathcal K_{(x_0,v_0)}x:=\nabla_{\dot x}v|_{t=0}=0.$ Secondly, the Leibniz rule gives 
$$\nabla_{\dot x}(S(x)v)|_{t=0}=S(x_0)\nabla_{\dot x}v+\dot x (S(x))|_{t=0}v_0=\dot x (S(x_0))v_0, $$
where $\dot x (S(x_0))$ denotes the matrix obtained by applying $\dot x$ on each entry. Finally,
$$W \pi_2^*S(x_0,v_0)=g_x((d\pi(x_0,v_0)W)S(x)v,v)=(d\pi(x_0,v_0)W)\pi_2^*S(x_0,v_0).$$
In particular, we obtain 
$$-\sum_{j=1}^ne_j(e_j(S_x(v,v)))=-\sum_{j=1}^nX_jX_j\pi_2^*S. $$
This is exactly \eqref{eq:deltaH} as we wanted to show.
\end{proof}
\subsubsection{Symmetrized covariant derivative and geodesic vector field.}
\label{X+X-}
The vector field $X$ acts as a Lie derivative on $C^{\infty}(SM)$ and moreover, for any $k\geq 0$:
$$X:C^{\infty}(M,\Omega_k)\to C^{\infty}(M,\Omega_{k-1})\oplus C^{\infty}(M,\Omega_{k+1}). $$
We will write $X=X_++X_-$, where $X_{\pm}:C^{\infty}(M,\Omega_k)\to C^{\infty}(M,\Omega_{k\pm1})$, see \cite[Lemma 14.2.1]{Lef}.

The Levi-Civita connection acts naturally on $m$-tensors but it does not preserve the symmetry. We thus introduce the symmetrized covariant derivative: 
$$D_{g_0}:=\mathrm{Sym}\circ \nabla_{g_0}: C^{\infty}(M;S^mT^*M)\to C^{\infty}(M;S^{m+1}T^*M).$$
Its formal adjoint is the divergence operator $D_{g_0}^*$:
$$D_{g_0}^*=-\tr \circ \nabla_{g_0}: C^{\infty}(M;S^mT^*M)\to C^{\infty}(M;S^{m-1}T^*M). $$
We have the important relation, see \cite[Lemma 14.1.8]{Lef}:
\begin{equation}
\label{eq:XandD}
X\pi_m^*S=\pi_{m+1}^*D_{g_0}S, \ \forall S\in C^{\infty}(M;S^mT^*M).
\end{equation}
\subsubsection{Pestov's identity}
We introduce an "energy estimate" due to Pestov (and first written in coordinate-free way in \cite[Appendix Theorem 1.1]{Kni}, see also \cite[Corollary 3.6]{PSU}). It will be crucial in the proof of solenoidal injectivity of Proposition \ref{sinj}. More precisely, we will need the following estimate which is implied in negative curvature by the localized Pestov identity. We will use the result of \cite[Theorem 14.3.4, Equation (16.1.3)]{Lef}, which is valid in dimension $n\geq 4$:
\begin{equation}
\label{eq:Pestov}
\forall n\geq 4, \quad\forall u\in \Omega_m,\ \|X_-u\|_{L^2(SM)}^2\leq \|X_+u\|_{L^2(SM)}^2.
\end{equation}
A consequence is that the operator $X_+$ is injective on $\Omega_m $ for $m\geq 1$ (\cite[Lemma 16.1.6]{Lef}):
\begin{equation}
\label{eq:X_+inj}
\forall u\in \Omega_m, \ m\geq 1, \ X_+u=0\Rightarrow u=0.
\end{equation}
Finally, let $S \in C^{\infty}(M;S^2T^*M)$ be a divergence-free tensor, i.e., $D_{g_0}^*S=0$ and write $S=S_0+L(h)$ with $S_0\in C^{\infty}(M;S_0^2T^*M)$ and $h\in C^{\infty}(M)$. Then 
$$D_{g_0}^*S=0\iff D_{g_0}^*(S_0+L(h))=0\iff D_{g_0}^*S_0-D_{g_0}h=0, $$
where we have used the relation $D_{g_0}^*L(h)=D_{g_0}h$ on $0$-tensors. The divergence operator corresponds to a rescaled version of $X_-$. Indeed, using \eqref{eq:XandD} and \eqref{eq:conformal}, one has, for any $S_0\in C^{\infty}(M;S^2_0T^*M)$ and $p\in C^{\infty}(M;T^*M)$,
$$\frac{\Gamma(n/2+1)}{\pi^{n/2}}\langle \pi_1^*(D_{g_0}^*S),\pi_1^*p\rangle=\langle D_{g_0}^*S_0,p\rangle=\langle S,D_{g_0}p\rangle=\frac{\Gamma(n/2+2)}{\pi^{n/2}}\langle \pi_2^*S,X_+\pi_1^*p\rangle. $$
In particular, this gives 
\begin{equation}
\label{eq:div}
 \pi_1^*(D_{g_0}^*S_0)=-\tfrac{n+2}2X_-\pi_2^*S_0.
 \end{equation}
  This means that
\begin{equation}
\label{eq:divfree}
D_{g_0}^*S=0\iff \tfrac{n+2}2X_-\pi_2^*S_0=-X_+h.
\end{equation}

\subsubsection{Decomposition of the space of symmetric tensors.}
The proof of Theorem \ref{MainTheo} follows from a Taylor expansion of the functional $\Phi$ near a locally symmetric metric $g_0$. We note that there exists a natural gauge given by the action of the group $\mathrm{Diff}_0(M)$ of diffeomorphisms isotopic to the identity. We define
$$\mathcal O(g):=\{\phi^*g\mid \phi \in \mathrm{Diff}_0(M)\},\ T_{g_0}\mathcal O(g_0)=\{\mathcal L_Vg_0\mid V\in C^{\infty}(M;TM)\}. $$
In particular, the functional $\Phi$ is constant along any orbit $\mathcal O(g)$ and the kernel of the Hessian $d^2\Phi_{g_0}$ always contains the tangent space $T_{g_0}\mathcal O(g_0)$. We will hence need to prove an injectivity result of the Hessian on a "transverse slice" to $T_{g_0}\mathcal O(g_0)$. We remark the following fact:
$$T_{g_0}\mathcal O(g_0)=\{D_{g_0}p\mid p\in C^{\infty}(M;TM)\}, $$
in particular, a natural transverse slice is provided by the kernel of the adjoint $D_{g_0}^*$, \cite[Theorem 14.1.10]{Lef}.
\begin{lemm}
\label{lemmD*}
For any $S\in C^{\infty}(M;S^mT^*M)$, there exists a unique pair 
$$(p,h)\in C^{\infty}(M;S^{m-1}T^*M)\times \big(C^{\infty}(M;S^mT^*M)\cap \mathrm{Ker}(D_{g_0}^*)\big)$$
such that $S=D_{g_0}p+h$.
\end{lemm}
The analysis will take place on the space of divergence-free (or solenoidal) tensors $C^{\infty}(M;S^mT^*M)\cap \mathrm{Ker}(D_{g_0}^*)$. We thus recall the following lemma which allows to "project" a metric $g$ on solenoidal tensors. It was obtained in this form in \cite[Lemma 2.4]{GuKnLef} but the idea goes back to Ebin \cite{Eb}.
\begin{lemm}[Slice lemma]
\label{slice lemma}
Let $k\geq 2$,  and $\alpha\in (0,1)$. Then there exists a neighborhood $\mathcal U$ of $g_0$ in the $C^{k,\alpha}$-topology such that for any $g\in \mathcal U$, there is a unique $\phi \in \mathrm{Diff}_0(M)$ of regularity $C^{k+1,\alpha}$, close to identity, such that $\phi^*g\in \mathrm{Ker}(D_{g_0}^*)$ is divergence free. Moreover, there exists $\epsilon>0$ and $C>0$ such that
$$\|g-g_0\|_{C^{k,\alpha}}\leq \epsilon \ \Rightarrow \ \|\phi^*g-g_0\|_{C^{k,\alpha}}\leq C\|g-g_0\|_{C^{k,\alpha}}.$$

\end{lemm}
\begin{figure}
{\includegraphics[width=12cm]{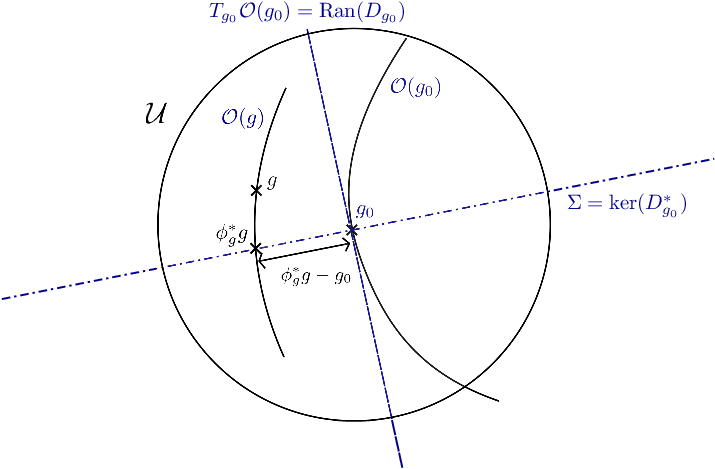}}
\caption{An illustration of the slice Lemma \ref{slice lemma} in a neighborhood $\mathcal U$ of $g_0$.}
\end{figure}
\subsection{The operator $\Pi$}
\label{Xray}
We recall in this subsection the main feature of the "generalized X-ray transform" $\Pi$. It was first introduced by Guillarmou in \cite{Gu} and used by Guillarmou and Lefeuvre in \cite{GuLef} for their proof of the local injectivity of the marked length spectrum. A quick definition of the operator can be given using the variance.
Let $f$ have zero mean, i.e., $\int_{SM}f(z)d\mu_{\mathrm{Liou}}^g(z)=0$. The variance of $f$, with respect to the Liouville measure $\mu^g_{\mathrm{Liou}}$ is
\begin{equation}
\label{eq:variance}
\mathrm{Var}_{\mu^g_{\mathrm{Liou}}}(f):=\lim_{T\to +\infty}\frac 1 T\int_{SM}\left(\int_0^Tf(\varphi_t(z))dt\right)^2d\mu_{\mathrm{Liou}}^g(z).
\end{equation}
Then for any smooth function $f\in C^{\infty}(SM)$, we define
\begin{equation}
\label{eq:Pi}
\quad \langle \Pi f,f\rangle=\int_{\mathbb R}\langle f\circ \varphi_t,f\rangle_{L^2}dt=\mathrm{Var}_{\mu_{\mathrm{Liou}}^g}(f).
\end{equation}
The second equality follows from the exponential decay of correlations \cite{Liv}, see for instance \cite[Equation $(2.6)$]{GuKnLef}. The Hessian $d^2\Phi(g_0)$ will be rewritten using the operator $\Pi$ in Theorem \ref{theoHess}.

A more microlocal definition of the operator $\Pi$ was given originally in \cite{Gu} which allowed to track its analytical properties but we will not need to recall this definition here. We now define 
\begin{equation}
\label{eq:Pim}
\Pi_m:=(\pi_m)_*(\Pi+\langle \cdot, 1\rangle_{L^2})\pi_m^*.
\end{equation}
The operator $\Pi_m$ was shown to be a pseudo-differential operator of order $-1$ which is elliptic and injective on solenoidal tensors (see \cite[Theorem 3.5, Lemma 4.3]{Gu,GuLef}). For any $Q\in C^{\infty}(M;S^mT^*M)$, using \eqref{eq:Pi} gives
\begin{equation}
\label{eq:Pi2}\mathrm{Var}_{\mu_{\mathrm{Liou}}^{g_0}}(\pi_m^*Q,\pi_m^*Q)=\langle \Pi_2^{g_0}Q,Q\rangle -\langle \pi_2^*Q,1\rangle_{L^2(SM)}^2.
\end{equation}
In negative sectional curvature, we have the following coercive estimate, which we will use crucially to obtain the stability estimate \eqref{eq:stability} \cite[Lemma 2.1]{GuKnLef}.
\begin{equation}
\label{eq:coercive}
\exists C>0,\ \forall h\in H^{-1/2}(M;S^mT^*M),\ \langle \Pi_m h,h\rangle \geq C\|\Pi_{\mathrm{Ker}(D_{g_0}^*)}h\|_{H^{-1/2}}^2
\end{equation}
where $\Pi_{\mathrm{Ker}(D_{g_0}^*)}$ is the projection on solenoidal (or divergence-free) tensors.

\subsection{Locally symmetric spaces}
\label{locsym}
Katok's entropy conjecture concerns locally symmetric metrics. We recall here the Riemannian definition. Recall that given a metric $g$, its $(3,1)$ curvature tensor $R_g$ is defined to be
\begin{equation}
\label{eq:Rg} R_g(X,Y,Z):=\nabla_{[X,Y]}Z-[\nabla_X,\nabla_Y]Z, \qquad X,Y,Z\in C^{\infty}(M;TM).
\end{equation}
\begin{defi}
Let $M^n$ be a differentiable manifold. A metric $g$ on $M$ is locally symmetric if its curvature tensor $R_g$ is parallel : $\nabla_g R_g=0$.
\end{defi}
If $(x,v)\in SM$ and $Y$ is the corresponding geodesic vector field, then one can find parallel orthonormal fields $Y_1=Y,Y_2,\ldots, Y_n$ (along the geodesic defined by $(x,v)$) such that $R(Y,Y_j)Y=-\lambda_j^2Y_j$ for some $\lambda_j$ that do not depend on $(x,v)$. 

This means that the unstable Jacobian $J^u(v):=-\tfrac{d}{dt}|_{t=0}\mathrm{det}(d\varphi_t(v)|_{E_u(v)})$, where $E_u(v)$ is the unstable bundle at  $v\in SM$, is a constant function of $v$. In particular, its equilibrium state (the Sinai-Ruelle-Bowen measure, which in this case is just the Liouville measure) coincides with the equilibrium state for the null-potential (i.e., the measure of maximal entropy). This provides the first direction of Katok's entropy conjecture, namely: for a locally symmetric metric, the Liouville measure is the measure of maximal entropy.

There is a more algebraic characterization of locally symmetric spaces. Indeed, it is known that \cite[Chapter 19]{Most} compact locally symmetric spaces of negative sectional curvature are obtained as compact quotients of hyperbolic spaces. More precisely, one has

If $(M^n,g_0)$ is a closed manifold isometric to a locally symmetric space of negative sectional curvature, then $(M^n,g_0)$ is a compact quotient of one of the following spaces:
\begin{itemize}
\item The real hyperbolic space $\mathbb R \mathbb H^n$.
\item The complex hyperbolic space $\mathbb C \mathbb H^{n/2}$.
\item The quaternionic hyperbolic space $\mathbb H \mathbb H^{n/4}$.
\item The octonionic hyperbolic plane $\mathbb O \mathbb H^2$ and in this case the real dimension of $M$ is equal to $16$.
\end{itemize}

We note that for a locally symmetric metric, the Lyapunov exponents of the geodesic flow are given by:
\begin{itemize}
\item $(1,1,\ldots,1, 0, -1,\ldots, -1)$ for a quotient of $\mathbb R \mathbb H^n$.
\item $(2,1,1,\ldots,1, 0, -1,-1,\ldots,-1, -2)$ for a quotient of $\mathbb C \mathbb H^{n/2}$
\item $(2,2,2,1,\ldots,1, 0, -1,-1,\ldots,-2,-2, -2)$ for a quotient of $\mathbb H \mathbb H^{n/4}$.
\item $(\underbrace{2,2,\ldots,,2}_{\text{seven times}},1,\ldots,1, 0, -1,-1,\ldots,\underbrace{-2,\ldots,-2}_{\text{seven times}})$ for a quotient of $\mathbb O \mathbb H^2$.
\end{itemize}
In particular, we have 
\begin{equation}
\label{eq:prev} g_0(R(Y_j,Y)Y_j,Y)=g_0(R(Y,Y_j)Y,Y_j)=-g_0(\lambda_j^2 Y_j,Y_j)=-\lambda_j^2\in\{-1,-4\}
\end{equation}
where $Z\mapsto R(Y_j,Z)Y_j$ is a symmetric endomorphism. Moreover, since one has $-4\leq g(R(Y_j,Z)Y_j,Z)\leq-1$ for any unit length $Z$, this implies that
\begin{equation}
\label{eq:Lyapunov}
\forall j=1,\ldots, n, \ \nabla_Y Y_j=0, \ R(Y,Y_j)Y=-\lambda_j^2Y_j, \quad R(Y_j,Y)Y_j=-\lambda_j^2Y.
\end{equation}

In the rest of this section, we recall some geometric results we will need about locally symmetric spaces.

\subsubsection{Complex hyperbolic quotient} 
Since $(M^n, g)$ is a quotient of the complex hyperbolic space, it is equipped with a complex structure. By this, we mean that there exists a (global) section $J\in C^{\infty}(M;\mathrm{End}(TM))$ satisfying $J^2=-\mathrm{Id}$, compatible with the metric structure, i.e., $g(\cdot,\cdot)=g(J\cdot,J\cdot)$. Moreover, $(M^n ,g)$ is a Kähler manifold, which means that $J$ is parallel:
\begin{equation}
\label{eq:J//}
\nabla J=0, \qquad  \forall X,Y \in C^{\infty}(M,TM), \ \nabla_X (JY)=J(\nabla_X Y).
\end{equation}

The complex hyperbolic space has constant holomorphic sectional curvature:
\begin{equation}
\label{eq:v,Jv}
\forall (x,v)\in SM, \quad K_x(v,Jv):=g_x(R_x(v,Jv)v,Jv)=-4,
\end{equation}
see \cite[Chapter IX]{KN}. In particular, the parallel vector field $Y_2$ appearing in \eqref{eq:Lyapunov} can be taken equal to $JY$ in this case.

For a symmetric tensor $T\in C^{\infty}(M;S^mT^*M)$, we define $(T\circ J):=T(J\cdot,\ldots,J\cdot). $ Using the pullback map $\pi_m^*$, it corresponds to a natural action of $J$ on $C^{\infty}(SM)$:
\begin{equation}
\label{eq:J1}
\pi_m^*(T\circ J)=J(\pi_m^*T), \ (Jf)(x,v):=f(x,Jv), \ (x,v)\in SM,\  f\in C^{\infty}(SM). 
\end{equation}
Since $J$ commutes with the trace, we obtain using \eqref{eq:iso} that $J$ preserves $\Omega_m$ for any $m\in \mathbb N$.
For later use, we also define a twisted version of the symmetrized covariant derivative:
\begin{equation}
\label{eq:D^J1}
\forall S\in C^{\infty}(M;S^mT^*M), \ D_{g_0}^JS:=(D_{g_0}(S\circ J))\circ J.
\end{equation}

\subsubsection{Quaternionic hyperbolic quotient.}
Quotients of the hyperbolic space $\mathbb H \mathbb H^n$ are examples of a quaternion-Kähler manifolds, \cite[Chapter 14]{Bes}. This time, there is no global complex structure but locally, one can find three almost complex structures $J_1,J_2,J_3:=J_1J_2$ compatible with the metric $g_0$. The endomorphism $J_i$ are not defined on the whole manifold but
$$Q:=\mathrm{Span}(J_1,J_2,J_3)$$
is a globally well-defined vector bundle. We construct a natural metric on $Q$ by requiring $(J_1,J_2,J_3)$ to be orthonormal. We will denote by $Z\subset Q$ the unit ball for this metric.  We remark that for a quotient of the quaternionic hyperbolic space, one has :
\begin{equation}
\label{eq:curvquater}
\forall (x,v)\in SM,\ \forall J\in Z, \quad K_x(x,Jv):=g_x(R_x(v,Jv)v,Jv)=-4, 
\end{equation}
see \cite[Proposition 10.12]{BriHa} or \cite[(14.44)]{Bes}. Note that it is not clear in this case than one can take $Y_2=J_1Y,Y_3=J_2Y$ and $Y_4=J_3Y$ in \eqref{eq:Lyapunov} as the almost complex structure $J_i$ may not be parallel but one still has
\begin{equation}
\label{eq:vect}
\mathrm{Span}(Y_2,Y_3,Y_4)=\mathrm{Span}(J_1Y,J_2Y,J_3Y).
\end{equation}
Similarly to the case of a quotient of a complex hyperbolic space, we will write
\begin{equation}
\label{eq:D^J}
\forall S\in C^{\infty}(U;S^mT^*U), \ D_{g_0}^{J_i}S:=(D_{g_0}(S\circ J_i))\circ J_i
\end{equation}
where $U$ is an open set on which $J_i$ are defined and span $E$.
\subsubsection{Octonionic Cayley plane}
This is the last possible type of rank one locally symmetric space of negative curvature. Note that it can only happen in dimension $16$. We refer to the appendix of \cite{Ruan} or \cite[Chapter 19]{Most} for more details about the construction. From the remark after \cite[Corollary A.6]{Ruan}, one sees that there may not exist local endomorphisms $J_1,\ldots, J_7\in \mathrm{End}(TM)$ such that $K_x(v,J_iv)=-4$ and we will thus use a slightly different approach. Let $\mathbb O$ denote the set of octonions. We define the \emph{Cayley line} through $(a,b)\in \mathbb O^2\setminus\{0\}$ to be
\begin{equation}
\label{eq:Cayleyline}
\mathrm{Cay}(a,b)=\begin{cases}\mathbb O(1,a^{-1}b) \text{ if } a\neq 0,
\\ \mathbb O(0,1) \text{ if } a= 0.
\end{cases}
\end{equation}
By \cite[Proposition A.5]{Ruan}, there exists an isomorphism $\Psi_x:\mathbb O^2\to T_xM$ for any point $x\in M$. We define the Cayley line of a vector to be
\begin{equation}
\label{eq:Cayley}
\mathrm{Cay}(v):=\Psi_x(\mathrm{Cay}(\Psi_x^{-1}(v))), \quad \forall v\in T_xM\setminus\{0\}.
\end{equation}
From \cite[Corollary 5.5]{Ruan}, we get, with the notations of \eqref{eq:Lyapunov}, for any $(x,v)\in SM$:
\begin{equation}
\label{eq:CayleySpan}
\mathrm{Cay}(v)=\mathrm{Span}(Y_1(x,v),\ldots,Y_8(x,v)),\ \forall z\in \mathrm{Cay}(v)\setminus \mathbb R v, \ K_x(v,z)=-4.
\end{equation}
\subsection{A theory of raising and lowering operators for complex hyperbolic spaces.}
\label{CH}
In this subsection, we build an analogous theory of the Guillemin–Kazhdan \cite{GK} raising/lowering
operators for surfaces for complex hyperbolic spaces. For a textbook account of the theory for surfaces, see \cite[Chapter 15.2]{Lef} or \cite[Chapter 1.11]{GuMaz}. The key point is that the complex structure recalled in \eqref{eq:J//} allows for a further splitting of the space of spherical harmonics. In particular, the geodesic vector field $X$ splits into raising and lowering operators which have specific mapping properties with respect to this splitting. In Lemma \ref{lemmV}, we show that these operators are injective when restricted to certain spaces and this will be the key of the proof of Theorem \ref{soltheo} in Section \ref{sinj}.

We define two important vector fields and compute their commutation relations. Consider a local orthonormal basis $(e_1,Je_1,e_2,Je_2, \ldots,e_m,Je_m)$ where $m=n/2$ near a point $p$. For $v\in T_xM$, we will write $v=\sum_{i=1}^m v_ie_i+v_{i+m}Je_i$. We remark that in these coordinates, one has $ Jv=\sum_{i=1}^m -v_{i+m}e_i+v_{i}Je_i$ where the indices for $v_i$ are taken modulo $m$.

For $(x,v)\in TM$, the vertical space $\mathbb V(x,v)$ is spanned by $Y_i(x,v):=\mathcal K_{(x,v)}^{-1}(e_i(x))$ and $ Y_{i+m}(x,v):=\mathcal K_{(x,v)}^{-1}(Je_i(x))$. They are called the vertical lifts of the basis. The horizontal space $\mathbb H(x,v)$ is spanned by $X_i(x,v):=d\pi(x,v)^{-1}(e_i(x))$ and $ X_{i+m}(x,v):=d\pi_{(x,v)}^{-1}(Je_i(x))$ and they are the horizontal lifts of the basis. We define:
\begin{equation}
\label{eq:Xperp}
H(x,v)=d\pi(x,v)^{-1}(-Jv)=\sum_{i=1}^m v_{i+m}X_i(x,v)-v_i  X_{i+m}(x,v)
\end{equation}
as well as
\begin{equation}
\label{eq:V}
V(x,v)=\mathcal K(x,v)^{-1}(Jv)=\sum_{i=1}^m -v_{i+m}Y_i(x,v)+v_i  Y_{i+m}(x,v).
\end{equation}
Denote by $J^{-1}=J^3$ the inverse of $J$ defined in \eqref{eq:J1}. We have the following lemma which gives an explicit link between $H$ and $X$.
\begin{lemm}[Conjugation relation]
The following holds for any $S\in C^{\infty}(M;S^mT^*M)$ and any $u\in C^{\infty}(SM)$,
\begin{equation}
\label{eq:H=JXJ}
 Hu=JXJ^{-1}u,  \qquad \pi_{m+1}^*(D_{g_0}^JS)=(-1)^mH\pi_m^*S.
\end{equation}
In particular for $m\geq 0$, one gets $J^{-1}=(-1)^mJ$ on $\Omega_m$. One has
$$H:C^{\infty}(M,\Omega_k)\to C^{\infty}(M,\Omega_{k-1})\oplus C^{\infty}(M,\Omega_{k+1}). $$
We will write $H=H_++H_-$, where $H_{\pm}:C^{\infty}(M,\Omega_k)\to C^{\infty}(M,\Omega_{k\pm1})$. On $\Omega_m$,
\begin{equation}
\label{eq:Hconj}
H=(-1)^mJXJ, \quad H_{\pm}=JX_{\pm}J^{-1}=(-1)^mJX_{\pm}J.
\end{equation}
In particular, one has a Pestov inequality for $H$:
\begin{equation}
\label{eq:PestovH}
\forall n\geq 4, \quad \forall u\in \Omega_m,\ \|H_-u\|_{L^2(SM)}^2\leq \|H_+u\|_{L^2(SM)}^2.
\end{equation}
\end{lemm}
\begin{proof}
Consider $f\in C^{\infty}(SM)$. For any horizontal vector field $W$, choose a path $(x(t),v(t))$ in $SM$ such that $(x(0),v(0))=(x,v)\in T(SM)$ and $(\dot x(0),\dot v(0))=W$. Recall that the fact that $W$ is horizontal is equivalent to $\mathcal K_{(x,v)}W:=\nabla_{\dot x}v|_{t=0}=0.$ Using the expansion in spherical harmonics and the isomorphism \eqref{eq:pi_k^*}, it is sufficient to prove \eqref{eq:H=JXJ} for $f=\pi_m^*S$ with $S\in  C^{\infty}(M;S_0^mT^*M)$ a spherical harmonics of degree $m$. In this case, one gets 
\begin{align*}W(Jf)(x,v)=&\frac{d}{dt}|_{t=0}f\big(x(t),Jv(t)\big)=\frac{d}{dt}|_{t=0}\sum_{i_1,i_2,\ldots ,i_m}S^{i_1,\ldots i_m}(x(t))\prod_{j=1}^m(Jv)_{i_j}(t),
\end{align*}
for some locally defined functions $ S^{i_1,\ldots i_m}$. In the last equation, we have denoted by $v_{j}(t)=g_{x(t)}(v(t),e_j(t))$ the coefficients of $v$ in some orthonormal moving frame that we can choose to be parallel at $(x,v)$. Using the fact that the $e_i$ are parallel at $(x,v)$ and the fact that $\mathcal K_{(x,v)}W:=\nabla_{\dot x}v|_{t=0}=0$ yields
$$ \frac{d}{dt}|_{t=0}g_{x(t)}(v(t),e_j(t))=g_x(\nabla_{\dot x}v(0),e_j(0))+g_x(v,\nabla_{\dot x}e_j)=0.$$
Using Leibniz rule then gives
$$W(Jf)(x,v)=\frac{d}{dt}|_{t=0}\left(\sum_{i_1,i_2,\ldots ,i_m}S^{i_1,\ldots, i_m}(x(t))\right)\prod_{j=1}^m(Jv)_{i_j}=[Wf](x,Jv). $$
Applying the previous result for $W=X_i$ and plugging this into \eqref{eq:Xperp} yields \eqref{eq:H=JXJ}.
Moreover, using the fact that $J$ is an isometry and \eqref{eq:Pestov} gives \eqref{eq:PestovH}. Equation \eqref{eq:Hconj} follows from \eqref{eq:H=JXJ}, the observation that $J^3=(-1)^mJ$ on $\Omega_m$ and the fact that $J$ preserves the space $\Omega_m$.
\end{proof}
 From \eqref{eq:V}, we see that the action of $V$ on $\Omega_m$ corresponds to 
\begin{equation}
\label{eq:VOmega}
\forall T\in C^{\infty}(M;S^mT^*M), \ V\pi_m^*T=\pi_m^*(VT),\quad VT:=m\times \mathrm{Sym}(T(J\cdot,\cdot,\ldots,\cdot)). 
\end{equation}
A direct computation, using $\sum_i T(Je_i,e_i,\ldots)=-\sum_iT(e_i,Je_i,\ldots)$, gives
$$ \forall T\in C^{\infty}(M;S^mT^*M), \ \tr(V T)=V \tr(T),$$
which show that $V$ stabilizes $\Omega_m$ for any $m$. 

The vector fields $H$ and $V$ are related to the geodesic flow by the following relations.
\begin{lemm}[Commutation relations for complex case]
\label{lemmcommComplex}
The following identities hold:
\begin{align}
\label{eq:comm}
\quad [X,V]=H,&\quad [V,H]=X, \quad [H,X]=-4V.
\\
\label{eq:comm2}
[H_+,X_+]=[X_-,H_-]=0,&\quad  H_+X_-+H_-X_+-X_+H_--X_-H_+=-4V.
\end{align}
\end{lemm}
The proof is the same as in the case of surfaces (see \cite[Proposition 1.51]{GuMaz}). The factor $-4$ comes from the fact that for any unit-length vector $v\in T_xM$, one has $K_x(v,Jv)=-4$, where $K_x$ denotes here the sectional curvature, see \eqref{eq:v,Jv}.

We now analyse the action of $H$ and $X$ on the eigenspaces of $V$. We define raising and lowering operators acting on eigenspaces of $V$
\begin{equation}
\label{eq:eta}
\eta_{+}^{\pm}=\frac {X_+   \pm  iH_+}{2}, \quad \eta_{-}^{\pm}=\frac {X_-    \pm i H_-}{2} .
\end{equation}
In the cases of surfaces, these operators were introduced and studied by Guillemin and Kazhdan, see \cite[Section 3]{GK}.
\begin{lemm}[Eigenspaces of $V$]
\label{lemmV} For any $m\geq 0,$ the restriction $V_{|\Omega_m}$ is antisymmetric and thus has spectrum included in $i\mathbb R$. Moreover, one has\footnote{The previous equality should hold for any $m$ but since we do not need it, we will not check it.}
$$\mathrm{Spec}(V_{|\Omega_m})=\{i(m-2k) \mid 0\leq  k \leq m\},\quad 0\leq m \leq 3. $$
For $\lambda \in \mathbb Z$, define $E^\lambda _m:=\Omega_m\cap \mathrm{ker}(V-i\lambda)$, then one has
\begin{equation}
\label{eq:XandH}
X_{\pm}=\eta^{+}_{\pm}+\eta^{-}_{\pm}, \quad 
H_{\pm}=i(\eta^{+}_{\pm}-\eta^{-}_{\pm}).
\end{equation}
Here, the raising and lowering operators have the mapping property $\eta^{\delta}_{\epsilon}: E^\lambda_m\to E^{\lambda+\delta}_{m+\epsilon}$ where $\epsilon,\delta\in \{\pm 1\}$. Finally, if $\lambda> 0$, then $\eta_+^+: E^\lambda _m \to E^{\lambda+1}_{m+1}$ is injective.

\end{lemm}
\begin{proof}
On $\Omega_0$, $V$ acts as the zero-operator so its spectrum is equal to $\{0\}$. On $\Omega_1$, $V$ acts as $J$ and using $J^2=-\mathrm{Id}$, one sees that the spectrum is given by $\{\pm i\}$. For $\Omega_2$, we compute for any $ T\in C^{\infty}(M;S^2T^*M)$,
$ V^2T=-2T+2T\circ J, $
which implies
\begin{equation}
\label{eq:JandV}
\forall u\in \Omega_2, \ Ju=u+\frac 12 V^2u.
\end{equation}
In particular, we see that $J$ is a polynomial in $V$ and this implies that the spectrum of $V$ is $\{\pm 2i,0\}$. For $\Omega_3$, we get this time, for any $T\in C^{\infty}(M;S^3T^*M)$,
\begin{equation}
\label{eq:eqtriple}
 \begin{cases}V^2T=-3T-2V( T\circ J)
\\
V^3 T=-7VT+6T\circ J.
\end{cases}
\end{equation}
These two equations imply that the spectrum is $\{\pm 3i,\pm i\}$.

The mapping properties of $\eta^{\delta}_{\epsilon}$ is a consequence of the commutation relations \eqref{eq:comm}. 
Let us prove the injectivity statement. For this, we first note that $(\eta_+^+)^*=-\eta_-^-$ and we compute, using $X_+^*=-X_-$ and $H_+^*=-H_-$, 
\begin{align*}
4[(\eta_+^+)^*\eta_+^+-\eta_+^+(\eta_+^+)^*]&=-X_-X_++X_+X_--H_-H_++H_+H_-
\\&-i(X_-H_++X_+H_--H_-X_+-H_+X_-)
\\&=-X_-X_++X_+X_--H_-H_++H_+H_--4iV,
\end{align*}
where we used the commutation relation \eqref{eq:comm2}. In particular, if $u\in E^\lambda_m$ then
\begin{equation}
\begin{split}
\label{eq:NewPestov}
4\|\eta_+^+u\|_{L^2(SM)}^2&=4\|\eta_-^-u\|_{L^2(SM)}^2+\|X_+u\|_{L^2(SM)}^2-\|X_-u\|_{L^2(SM)}^2
\\&+\|H_+u\|_{L^2(SM)}^2-\|H_-u\|_{L^2(SM)}^2+4\lambda \|u\|_{L^2(SM)}^2. 
\end{split}
\end{equation}
By the Pestov identities \eqref{eq:Pestov}, \eqref{eq:PestovH}, one gets, as a consequence of the previous equation, $\|\eta_+^+u\|_{L^2(SM)}^2\geq \lambda\|u\|_{L^2(SM)}^2$. Since $\lambda> 0$, this concludes the proof of injectivity.
\end{proof}
We record the following commutation relations which follow from \eqref{eq:comm2}:
\begin{equation}
\label{eq:comm3}
[\eta_+^{+},\eta_+^-]=0,\quad [\eta_-^{+},\eta_-^-]=0.
\end{equation}

We will need other estimates which we record in the following lemma. For a symmetric tensor $T$, we write $T_0$ for its trace-free part.
\begin{lemm}
\label{lemm1}
Let $T_0\in C^{\infty}(M;S^m_0T^*M)$ and denote by $f=\pi_m^*T_0$. Then one has
$$\begin{cases}
\max(\|X_+f\|_{L^2(SM)}^2,\|H_+f\|_{L^2(SM)}^2)\leq (\Lambda_{m+1}^n)^{-1}\|\nabla_{g_0}T_0\|_{L^2(S^2T^*M)}^2,
\\
(\|Xf\|_{L^2(SM)}^2+\|Hf\|_{L^2(SM)}^2)\leq (\Lambda_2^n)^{-1}\|\nabla_{g_0}T_0\|_{L^2(S^2T^*M)}^2,
\end{cases} $$
where in the second bound, we have $m=2$ and $\Lambda_m^n$ is defined in \ref{eq:conformal}.
\end{lemm}
\begin{proof}
Using \eqref{eq:conformal}, one gets the following
\begin{align*}\|X_+f\|_{L^2(SM)}^2&=\|\pi_{m+1}^*(D_{g_0}T_0)_0\|_{L^2(SM)}^2=(\Lambda_{m+1}^n)^{-1}\|(D_{g_0}T_0)_0\|_{L^2(S^2T^*M)}^2
 \\&\leq (\Lambda_{m+1}^n)^{-1}\|D_{g_0}T_0\|_{L^2(S^2T^*M)}^2\leq (\Lambda_{m+1}^n)^{-1}\|\nabla_{g_0}T_0\|_{L^2(S^2T^*M)}^2.
\end{align*}
Here, we have used that $T\mapsto T_0$ and $\mathrm{Sym}$ are orthogonal projections. Similarly,
\begin{align*}\|H_+f\|_{L^2(SM)}^2&=\|\pi_{m+1}^*(D_{g_0}(T_0\circ J))_0\|_{L^2(SM)}^2=(\Lambda_{m+1}^n)^{-1}\|(D_{g_0}(T_0\circ J))_0\|_{L^2(S^2T^*M)}^2 \\&\leq (\Lambda_{m+1}^n)^{-1}\|D_{g_0}(T_0\circ J)\|_{L^2(S^2T^*M)}^2
\leq (\Lambda_{m+1}^n)^{-1}\|\nabla_{g_0}(T_0\circ J)\|_{L^2(S^2T^*M)}^2
\\&=(\Lambda_{m+1}^n)^{-1}\|\nabla_{g_0}T_0\|_{L^2(S^2T^*M)}^2,
\end{align*}
where we used that $J$ and $\nabla_{g_0}$ commute and that $J$ is an isometry. For the second bound, we remark that if we denote by $\nabla_{\mathbb H}^Q$ the projection of the horizontal gradient on $Q=\mathrm{Span}(X,H)\subset \mathbb H^{\mathrm{tot}}$, then  
\begin{align*}\|Xf\|_{L^2(SM)}^2+\|Hf\|_{L^2(SM)}^2&=\|\nabla_{\mathbb H}^Qf\|_{L^2(SM)}^2\leq \|\nabla_{\mathbb H}^{\mathrm{tot}}f\|_{L^2(SM)}^2=\langle \Delta_{\mathbb H}^{\mathrm{tot}}f,f\rangle _{L^2(SM)}\\&=\langle \pi_2^*(\nabla^*\nabla T_0), \pi_2^*T_0\rangle_{L^2(SM)}
=(\Lambda_2^n)^{-1}\|\nabla_{g_0}T_0\|_{L^2(S^2T^*M)}^2
\end{align*}
where we used  \eqref{eq:conformal} and \eqref{eq:deltaH}.
 \end{proof}
\subsection{Weitzenböck formula}
The last tool we will use is a Weitzenböck formula. Consider the symmetric $2$-tensor $S$ as $1$-form with values in $T^*M$. Denote by $d^{\nabla}$ the differential induced by the Levi Civita connection. Using \cite[(12.69)]{Bes} gives for any $ S\in C^{\infty}(M; S^2T^*M)$
\begin{equation}
\label{eq:Weit}
(d^{\nabla}d^{\nabla^*}+d^{\nabla^*}d^{\nabla})S=\nabla^*\nabla S-R^{\circ}(S)+S\circ \mathrm{Ric}.
\end{equation}
Here, we have used the notations
\begin{equation}
\label{eq:Ricci} R^{\circ}(S)(X,Y)=-\sum_{i=1}^n S(R(e_i,X)Y,e_i),\ S\circ \mathrm{Ric}(X,Y)=\sum_{i=1}^n S(R(e_i,X)e_i, Y)
\end{equation}
where $(e_i)_i$ is an orthonormal basis.
Since the curvature is negative, this identity provides a lower bound on the spectrum of the Laplacian, more precisely:
\begin{lemm}[Weitzenböck]
\label{Weitzenböck}
Let $(M^n,g_0)$ be a closed locally symmetric space of negative sectional curvature and let $S_0\in C^{\infty}(M; S_0^2T^*M)$.
\begin{itemize}
\item If $(M^n,g_0)$ is a quotient of the real hyperbolic space, then
\begin{equation}
\label{eq:Weit1}
\langle \nabla^*\nabla S_0, S_0\rangle_{L^2(S_0^2T^*M)}\geq n\|S_0\|_{L^2(S_0^2T^*M)}^2.
\end{equation}
\item If $(M^n,g_0)$ is a quotient of the complex hyperbolic space, then
\begin{equation}
\label{eq:Weit2}
\langle \nabla^*\nabla S_0, S_0\rangle_{L^2(S_0^2T^*M)}\geq (n+3)\|S_0\|_{L^2(S_0^2T^*M)}^2-3\langle S_0\circ J,S_0\rangle_{L^2(S_0^2T^*M)}.
\end{equation}
\item If $(M^n,g_0)$ is a quotient of the quaternionic hyperbolic space, then\footnote{Although the $J_i$ are only defined locally, the sum appearing in \eqref{eq:Weit3} is defined globally.}
\begin{equation}
\label{eq:Weit3}
\langle \nabla^*\nabla S_0, S_0\rangle_{L^2(S_0^2T^*M)}\geq (n+9)\|S_0\|_{L^2(S_0^2T^*M)}^2-3\sum_{i=1}^3\langle S_0\circ J_i,S_0\rangle_{L^2(S_0^2T^*M)}.
\end{equation}
\item If $(M^{16},g_0)$ is a quotient of the ocotonionic hyperbolic space, then
\begin{equation}
\label{eq:Weit4}
\langle \nabla^*\nabla S_0, S_0\rangle_{L^2(S_0^2T^*M)}\geq 36\|S_0\|_{L^2(S_0^2T^*M)}^2-\langle R^{\circ}(S_0),S_0\rangle_{L^2(S_0^2T^*M)}.\end{equation}
\end{itemize}
\end{lemm}
\begin{proof}
Except if stated otherwise, the notation $\langle \cdot,  \cdot\rangle$ will denote the scalar product on $L^2(S_0^2T^*M)$ in the following proof.
Suppose $S_0\in C^{\infty}(M; S^2T^*M)$ is trace-free, then equation \eqref{eq:conformal} gives
$$\langle R^{\circ}(S_0),S_0\rangle=\Lambda_2^n\langle \pi_2^*(R^{\circ}(S_0)), \pi_2^*S_0\rangle_{L^2(SM)}.$$
But if we consider parallel vector fields $Y_i$ such as in \eqref{eq:Lyapunov}
\begin{equation}
\label{eq:ww}
\begin{split}\pi_2^*(R^{\circ}(S_0))(x,v)&=-\sum_{i=1}^n (S_0)_x(Y_i,R(Y_i,Y)Y)
\\&=\begin{cases}
-\tr(S_0)+\pi_2^*S_0=\pi_2^*S_0
\\
-\tr(S_0)+\pi_2^*S_0-3\pi_2^*(S_0\circ J)=\pi_2^*S_0-3\pi_2^*(S_0\circ J)
\\ -\tr(S_0)+\pi_2^*S_0-3\sum_{i=1}^3S_0(Y_{i+1},Y_{i+1})
\\ -\tr(S_0)+\pi_2^*S_0-3\sum_{i=1}^7S_0(Y_{i+1},Y_{i+1})
\end{cases}
\end{split}
\end{equation}
where the different lines correspond to the real, complex, quaternionic and octonionic cases respectively.
For the quaternion hyperbolic case, one can use \eqref{eq:vect}, the fact that $(Y_2,Y_3,Y_4)$ and $(J_1Y,J_2Y,J_3Y)$ are orthonormal basis of the same space to get 
$$\pi_2^*S_0-3\sum_{i=1}^3S_0(Y_{i+1},Y_{i+1})=\pi_2^*S_0-3\sum_{i=1}^3S_0(J_iY,J_iY)=\pi_2^*S_0-3\sum_{i=1}^3\pi_2^*(S_0\circ J_i).$$
This computation also shows that the previous quantity does not depend on the choice of a basis $(J_1,J_2,J_3)$ and is thus globally defined. For the octonionic case, it is not clear that one can write $\sum_{i=1}^7S_0(Y_{i+1},Y_{i+1})=\pi_2^*T(x,v)$ where $T$ is some explicit function of $S$. Since we will not need the Weitzenböck formula for this case, we will not try to simplify this term further.
This then gives
\begin{align*}\langle R^{\circ}(S_0),S_0\rangle=\Lambda_2^n\langle \pi_2^*R^{\circ}(S_0),\pi_2^*S_0\rangle_{L^2(SM)}&=\begin{cases}
\|S_0\|^2
\\
\|S_0\|^2-3\langle S_0\circ J,S_0\rangle
\\ \|S_0\|^2-3\sum_{i=1}^3\langle S_0\circ J_i,S_0\rangle
\\ \langle R^{\circ}(S_0),S_0\rangle.
\end{cases}
\end{align*}
A similar argument gives
$$ \pi_2^*(S_0\circ \mathrm{Ric})(x,v)=\sum_{i=1}^nS_0(R(Y_i,Y)Y_i,Y)=-\left(\sum_{i=1}^n\lambda_i^2\right)\pi_2^*S_0=\begin{cases}
-(n-1)\pi_2^*S_0
\\
-(n+2)\pi_2^*S_0
\\ 
-(n+8)\pi_2^*S_0
\\
-36\pi_2^*S_0
\end{cases}.$$
This gives, using \eqref{eq:Lyapunov},
$$\langle S_0\circ \mathrm{Ric} ,S_0\rangle=\Lambda_2^n\langle \pi_2^*(S_0\circ \mathrm{Ric}), \pi_2^* S_0\rangle_{L^2(SM)}=\begin{cases}
-(n-1)\|S_0\|^2
\\
-(n+2)\|S_0\|^2
\\ 
-(n+8)\|S_0\|^2
\\
-36\|S_0\|^2
\end{cases}. $$
Using \eqref{eq:Weit} gives
\begin{align*} 0\leq \langle& \nabla^* \nabla S_0,S_0\rangle -\langle R^{\circ}(S_0),S_0\rangle +\langle S_0\circ \mathrm{Ric} ,S_0\rangle.
\end{align*}
This implies the desired bounds by adding the previous estimates.
\end{proof}

\section{Proof of the main theorem}
\label{secHessian}
\subsection{Re-writing the Hessian}
The first step consists in using the fact that $g_0$ is a locally symmetric metric to obtain an explicit formula for $d^2\Phi(g_0)$. The computation of the Hessian is done by Flaminio in \cite[Proposition 5.1.1]{Fla} for real hyperbolic metrics and we adapt it to the other classes of locally symmetric spaces. We prove the following theorem.
\begin{theo}[Hessian]
\label{theoHess}
Let $(M^n,g_0)$ be a closed, locally symmetric and negatively curved manifold. Then, for $S\in C^{\infty}(M;S^2T^*M)\cap \mathrm{Ker}(D_{g_0}^*)$, one has
\begin{equation}
\label{eq:theoHess}
\begin{split}
\langle d^2\Phi(g_0)S,S\rangle_{L^2(S^2T^*M)} &=\mathrm{Var}(\pi_4^*Q(S)).
\end{split}
\end{equation}
where $Q:C^{\infty}(M;S^2T^*M) \cap \mathrm{Ker}(D_{g_0}^*)\to C^{\infty}(M;S^{4}T^*M) $ is a differential operator. More precisely, with the notations introduced in Section \ref{preli},
\begin{itemize}
\item If $g_0$ is real hyperbolic, then
\begin{equation}
\label{eq:QH^n}
Q(S)=L\big(\frac 1 4 \nabla^* \nabla S+\frac 12 (\mathrm{tr}_{g_0} S)g_0-\frac 1 2 S\big).
\end{equation}
\item If $g_0$ is complex hyperbolic, then
\begin{equation}
\label{eq:QCH^n}
Q(S)=\frac 1 4 L(\nabla^* \nabla S)+\frac 1 2 \big({-}L(S)+ (\mathrm{tr}_{g_0} S)L(g_0)+L(S\circ J)\big)-\frac 1 8 D_{g_0}^JD_{g_0}^JS.
\end{equation}
\item If $g_0$ is quaternionic hyperbolic, then \footnote{Again, it is a statement of the theorem that the formula below is defined globally and not just locally.}
\begin{equation}
\label{eq:QHH^n}
Q(S)=\frac 1 4 L(\nabla^* \nabla S)+\frac 1 2 \big({-}L(S)+ (\mathrm{tr}_{g_0} S)L(g_0)+\sum_{i=1}^3L(S\circ J_i)\big)-\frac 1 8 \sum_{i=1}^3D_{g_0}^{J_i}D_{g_0}^{J_i}S.
\end{equation}
\item If $g_0$ is octonionic hyperbolic, then  
\begin{equation}
\label{eq:OH^n}
Q(S)=L\big(\frac 7{24}\nabla^*\nabla S+\frac 16 S+\frac 1 3 \tr(S)g_0-\frac 16 R^{\circ}(S)\big)-\frac 1 {24}\mathrm{Sym}(R^{\circ}(\nabla^2 S)) .
\end{equation}
\end{itemize}
\end{theo}
\begin{rem}
    Note that using \eqref{eq:Pi2}, the Hessian in \eqref{eq:theoHess} can be rewritten using the operator $\Pi_4$. This is what we will do in Subsection \ref{stab}.
\\ Note that for the real hyperbolic case, $Q(S)$ identifies with a symmetric $2$-tensor as 
$$\pi_4^*Q(S)=\pi_2^*\big(\frac 1 4 \nabla^* \nabla S+\frac 12 (\mathrm{tr}_{g_0} S)g_0-\frac 1 2 S\big).$$
This explains why the proof of the solenoidal injectivity is easier for this case.
\end{rem}
\begin{proof}
Consider a smooth path of metrics $(g_{\epsilon})$ and write $S=\partial_{\epsilon}g_{\epsilon}|_{\epsilon=0}$. 
First, we can identify all sphere bundles $S_gM$ by the pullback by $\Psi_g: SM\to S_gM$ defined by $\Psi_g(x,v):=(x,v/|v|_g)$. Next, for $\epsilon$ small enough, let $\Psi_{\epsilon}\in C^{\alpha}(SM,SM)$ be the \emph{normal $(g_\epsilon,g_0)$-Morse correspondance} (see \cite[Section 3.2]{Fla} for a precise definition) which gives an \emph{orbit equivalence}:
\begin{equation}
\label{eq:homeo}
\Psi_{\epsilon}(\varphi_t(z))=\varphi^{\epsilon}_{\kappa_{\epsilon}(z,t)}(\Psi_{\epsilon}(z))
%\quad d\Psi_{\epsilon}X(z)=a_{\epsilon}(z)X_{\epsilon}(\Psi_{\epsilon}(z))
\end{equation}
where the index $\epsilon$ indicates that we consider the corresponding objects for the metric $g_\epsilon$ and $\kappa_{\epsilon}$ is a time rescaling. 
Following Flaminio, we define
$$\bar \gamma : (-\epsilon_0,\epsilon_0)\times \mathbb R \mapsto  \bar\gamma (\epsilon, t):= \pi(\Psi_{\epsilon}(\varphi_tv)),$$
where $\pi:SM\to M$ is the projection.
This means that $\bar\gamma(\epsilon,\cdot)$ is a $g_{\epsilon}$ (non normalized) geodesic. We define $Y=\tfrac{d\bar \gamma}{dt}$ and $\xi=\tfrac{d\bar \gamma}{d\epsilon}$. Define $U(t,\epsilon):=U_{g_{\epsilon}}(\Psi_{\epsilon}(\varphi_tv))$ where $U_{g_\epsilon}(v)$ is the second fundamental form of the stable horosphere for the metric $g_\epsilon$ at $\pi_\epsilon(v)\in M$, seen a $(1,1)$-tensor. Note that, as already noticed in \cite[Page 582]{Fla} and with the notations of \eqref{eq:Lyapunov}, along a geodesic, one has $UY_j=-\lambda_jY_j$.
 Our starting point is \cite[Proposition A,  $(C_3)$, Lemma 3.3.11, Lemma 4.2.1]{Fla} which gives\footnote{We use Pesin's entropy formula to simplify the topological entropy appearing in \cite[Proposition A,  $(C_3)$]{Fla} with the sum of negative Lyapunov exponents appearing in \cite[Lemma 3.3.11]{Fla}}:
\begin{equation}
\label{eq:debut}
\frac{d^2}{d\epsilon^2}\Phi(g_{\epsilon})|_{\epsilon=0}=\mathrm{Var}_{\mu_{\mathrm{Liou}}}^{g_0}(\tr(\nabla_{\xi}U)).
\end{equation}
Now, the goal is to write more explicitly $\tr(\nabla_{\xi}U)$ as a function $V(S)$ of the infinitesimal variation $S=\partial_\epsilon|_{\epsilon=0}g_{\epsilon}$. 
We now consider parallel fields $Y_1=Y,Y_2,\ldots, Y_n$ as in \eqref{eq:Lyapunov}. 
The Riccati equation satisfied by $U_g$ translates into a differential equation satisfied on the diagonal entries $B_j:=\langle (\nabla_{\xi}U)Y_j,Y_j\rangle$. This is the content of \cite[Lemma 4.3.1]{Fla} and we have\footnote{We are using a different convention for the sign of $R$, hence the sign difference with \cite[Lemma 4.3.1]{Fla}.}  :
\begin{equation}
\label{eq:Bj}
\mathcal L_YB_j-2\lambda_j B_j=-S(Y,Y)\lambda^2_j-g_0\left(\frac{\partial R_{\epsilon}}{\partial\epsilon}|_{\epsilon=0}(Y,Y_j)Y,Y_j \right).
\end{equation}
We sum over $j$ to retrieve the trace:
\begin{equation}
\label{eq:eqdiff}
\sum_{j=1}^n B_j=\tr(\nabla_{\xi}U)\approx \frac 12\left(\sum_{j=2}^n \lambda_j\right)S(Y,Y)+\frac 12\sum_{j=1}^n \mu_j g_0\left(\frac{\partial R_{\epsilon}}{\partial\epsilon}|_{\epsilon=0}(Y,Y_j)Y,Y_j \right),
\end{equation}
where $\mu_1=1$ and $\mu_j=\lambda_j^{-1}$ for $j>1$.
Here, $\approx$ denotes the fact that the two functions are cohomologous, i.e., their difference is of the form $Xu$ where $X$ is the geodesic vector field associated to $g_0$ and $u\in C^{\infty}(SM)$ and we noted that the terms $\mathcal L_Y B_j$ are co-boundaries. We note that \eqref{eq:Bj} for $j=1$ gives that $g_0\left(\frac{\partial R_{\epsilon}}{\partial\epsilon}|_{\epsilon=0}(Y,Y)Y,Y \right)\approx 0$ (recall from \eqref{eq:Lyapunov} that $\lambda_1=0$) and we can thus add it in the right hand side. Moreover, we can compute the diagonal entry $B_1$:
$$B_1=\langle (\nabla_{\xi}U)Y,Y\rangle=\langle \xi \underbrace{U(Y)}_{=0}-U(\nabla_\xi Y),Y\rangle=-\langle U(\nabla_\xi Y),Y\rangle.  $$
Now, $0=\xi\langle Y,Y\rangle =2\langle Y, \nabla_\xi Y\rangle$. In other words, $\nabla_\xi Y $ is orthogonal to $Y$ and thus so is $U(\nabla_\xi Y)$ because the action of $U$ is diagonal. In particular, $B_1=0$ and it can be added to the left hand side without changing its value.

To understand the second term in the right hand side of \eqref{eq:eqdiff}, we need to compute the variation of the $(3,1)$ curvature
tensor defined in \eqref{eq:Rg} with respect to the metric. By \cite[Theorem 1.174]{Bes}, we have:
$$\frac{\partial R_{\epsilon}}{\partial\epsilon}|_{\epsilon=0}S(X,Y)Z=(\nabla_Y{\nabla'_g}S)(X,Z)-(\nabla_X{\nabla'_g}S)(Y,Z) $$
where $\nabla$ denotes the covariant derivative (for the metric $g_0$) on symmetric tensors and ${\nabla'_g}$ is the variation of the Levi Civita tensor which is characterized by
$$g_0({\nabla'_g}S(X,Y),Z):=g_0\left(\frac{\nabla^{g_{\epsilon}}_XY}{\partial\epsilon}|_{\epsilon=0}, Z \right)=\frac 12(\nabla_XS(Y,Z)+\nabla_YS(X,Z)-\nabla_ZS(X,Y)), $$
where $S=\partial_\epsilon|_{\epsilon=0}g_{\epsilon}.$
For the sake of notations, we will denote by $C_{g_0}$ the $(3,0)$ tensor field
$$C_{g_0}S(X,Y,Z):=g_0({\nabla'_g}S(X,Y),Z).$$ 
We can now compute the second term in \eqref{eq:eqdiff}:
\begin{align*}
\sum_{j=2}^n \mu_j g_0\left(\frac{\partial R_{\epsilon}}{\partial\epsilon}|_{\epsilon=0}(Y,Y_j)Y,Y_j \right)=\sum_{j=2}^n \mu_j\left(\nabla_{Y_j}(C_{g_0}S)(Y,Y,Y_j)-\nabla_{Y}(C_{g_0}S)(Y_j,Y,Y_j) \right).
\end{align*}
Using the fact that all fields are parallel, we get that $\nabla_{Y}(C_{g_0}S)(Y_j,Y,Y_j)$ is a coboundary and thus
$$ \sum_{j=1}^n \mu_j g_0\left(\frac{\partial R_{\epsilon}}{\partial\epsilon}|_{\epsilon=0}(Y,Y_j)Y,Y_j \right)\approx \sum_{j=1}^n \mu_j\nabla_{Y_j}(C_{g_0}S)(Y,Y,Y_j).$$
The first term can be expanded using the variation of the Levi Civita connection:
$$\nabla_{Y_j}(C_{g_0}S)(Y,Y,Y_j)=\nabla_{Y_j,Y}^2S(Y,Y_j)-\frac 12\nabla_{Y_j,Y_j}^2S(Y,Y),$$
where for a symmetric tensor $S$ and vector fields $X_1,X_2$, one has
\begin{equation}
\label{eq:nabla2}
    \nabla^2_{X_1,X_2} S:=\nabla_{X_1}(\nabla_{X_2}S)-\nabla_{(\nabla_{X_1}X_2)}S. \end{equation}
We use the Ricci identity \cite[Equation $(1.21)$]{Bes} and \eqref{eq:Lyapunov},  to re-write the first term:
\begin{align*} \nabla_{Y_j,Y}^2S(Y,Y_j)&=\nabla_{Y,Y_j}^2S(Y,Y_j)+S(R(Y_j,Y)Y,Y_j)+S(R(Y_j,Y)Y_j,Y)
\\&=\nabla_{Y,Y_j}^2S(Y,Y_j)-\lambda_j^2S(Y,Y).
\end{align*}
In total, one has using \eqref{eq:eqdiff}
\begin{align*}
V(S):=\tr(\nabla_{\xi}U)&= \frac 12\left(\sum_{j=2}^n \lambda_j\right)S(Y,Y)-\frac 1 4 \sum_{j=1}^n \mu_j\nabla_{Y_j,Y_j}^2S(Y,Y)\\&+\frac 1 2\sum_{j=1}^n \mu_j\nabla_{Y,Y_j}^2S(Y,Y_j)
+\frac 1 2\sum_{j=2}^n \lambda_jS(Y_j,Y_j)-\frac 1 2\sum_{j=2}^n\lambda_jS(Y,Y).
\end{align*}
That is, one has
\begin{equation}
\label{eq:V(S)}
V(S)=-\frac 1 2 \pi_2^*S-\frac 1 4 \sum_{j=1}^n \mu_j\left(\nabla_{Y_j,Y_j}^2S(Y,Y)-2\nabla_{Y,Y_j}^2S(Y,Y_j)\right)+\frac 1 2\sum_{j=1}^n \mu_j^{-1}S(Y_j,Y_j).
\end{equation}
\textbf{Quotient of $\mathbb R \mathbb H^n$.}
Suppose first that $g_0$ is a hyperbolic metric, then this means that $\mu_j\equiv 1$. In particular, one gets recalling \eqref{eq:star}
$$-\frac 1 4 \sum_{j=1}^n \mu_j\nabla_{Y_j,Y_j}^2S(Y,Y) =-\frac 1 4 \sum_{j=1}^n \nabla_{Y_j,Y_j}^2S(Y,Y)=\frac 1 4\pi_2^*(\nabla^*\nabla S),$$
as well as 
$$ -\frac 1 2 \sum_{j=1}^n \mu_j\nabla_{Y,Y_j}^2S(Y,Y_j)= -\frac 1 2 \sum_{j=1}^n \nabla_{Y,Y_j}^2S(Y,Y_j)=\frac 1 2 \pi_2^*(D_{g_0}D_{g_0}^*S)\approx0$$
using \eqref{eq:XandD}.
Finally, one has 
$$ \frac 1 2\sum_{j=1}^n \mu_j^{-1}S(Y_j,Y_j)=\frac 1 2\sum_{j=1}^nS(Y_j,Y_j)=\frac 1 2\tr(S).$$
Putting everything together gives $V(S)=\pi_4^*Q(S)$ where $Q$ is defined in \eqref{eq:QH^n}. Note that $Q$ is  $4$-tensor which identifies, using \eqref{eq:adjointt}, to the symmetric $2$-tensor $T$ used by Flaminio in \cite[Proof of Proposition 1.3.3]{Fla}.

\textbf{Quotient of $\mathbb C \mathbb H^n$.}
Suppose now that $g_0$ is a quotient of $\mathbb C \mathbb H^n$, then $\mu_j=1-\frac 1 2\delta_{j,2}$. Moreover, $Y_2$ is given by $Y_2=JY$ where $J$ is given by the complex structure. Notice that in contrast to the $Y_i$ that are only defined along the geodesic given by $(x,v)$, the operator $J$ is globally defined. We first see that 
$$ \frac 1 2\sum_{j=1}^n \mu_j^{-1}S(Y_j,Y_j)=\frac 1 2\sum_{j=1}^nS(Y_j,Y_j)+\frac 1 2 S(JY,JY)=\frac 1 2\tr(S)+\frac 12 \pi_2^*(S\circ J).$$
Next, we consider 
\begin{align*}-\frac 1 4 \sum_{j=1}^n \mu_j\nabla_{Y_j,Y_j}^2S(Y,Y) &=-\frac 1 4 \sum_{j=1}^n \nabla_{Y_j,Y_j}^2S(Y,Y)+\frac 1 8 \nabla^2_{JY,JY}S(Y,Y)
\\&=\frac 1 4\pi_2^*(\nabla^*\nabla S)+\frac 1 8 \nabla^2_{JY,JY}S(Y,Y),
\end{align*}
where we used \eqref{eq:D^J1}.
%Recall that
%$$\nabla^2_{JY,JY}S(Y,Y):=\nabla_{JY}(\nabla_{JY}S(Y,Y))-\nabla_{\nabla_{JY}JY}S(Y,Y). $$
%Now, we can use the fact that $J$ is parallel to obtain
%$$ \nabla_{JY}JY=J(\nabla_{JY}Y)=J(\nabla_{Y}JY+[JY,Y])=J(J\nabla_Y Y+ J[Y,Y])=0$$
%where we used the fact that $Y$ defines a geodesic.
 Now, using $J^2=-\mathrm{Id}$ gives
$$\nabla^2_{JY,JY}S(Y,Y)=\pi_4^*\big((\nabla \nabla (S\circ J))\circ J\big)=\pi_4^*\big(\mathrm{Sym}((\nabla \nabla (S\circ J))\circ J)\big). $$
The pre-composition by $J$ commutes with the symmetrization so 
$$\mathrm{Sym}((\nabla \nabla (S\circ J))\circ J)=\mathrm{Sym}(\nabla \nabla (S\circ J))\circ J. $$
Now, let $T\in C^{\infty}(M;S^2T^*M)$ be a symmetric $2$-tensor and let $X_1,X_2,X_3,X_4$ be tangent vectors, we compute:
\begin{align*} \mathrm{Sym}(\nabla \nabla T)(X_1,X_2,X_3,X_4)&=\frac 1 4 \sum_{i=1}^4 \nabla_{X_i}\mathrm{Sym}(\nabla T)(\hat X_i)=\frac 1 4 \sum_{i=1}^4 \nabla_{X_i}D_{g_0}T(\hat X_i)
\\&=\mathrm{Sym}(\nabla D_{g_0}T)= D_{g_0}D_{g_0}T(X_1,X_2,X_3,X_4).
\end{align*}
Here $\hat X_i$ denotes the $3$-tuple where we left out $X_i$. In particular one has, using \eqref{eq:XandD} and \eqref{eq:D^J1},
\begin{align*}-\frac 1 4 \sum_{j=1}^n \mu_j\nabla_{Y_j,Y_j}^2S(Y,Y) &=\frac 1 4\pi_2^*(\nabla^*\nabla S)+\frac 1 8\pi_4^*(D_{g_0}D_{g_0}(S\circ J)\circ J)
\\&=\frac 1 4\pi_2^*(\nabla^*\nabla S)-\frac 1 8\pi_4^*(D_{g_0}^JD_{g_0}^JS).
\end{align*}
The last term is given by 
\begin{align*}-\frac 1 2 \sum_{j=1}^n \mu_j\nabla_{Y,Y_j}^2S(Y,Y_j)&= -\frac 1 2 \sum_{j=1}^n \nabla_{Y,Y_j}^2S(Y,Y_j)-\frac 1 4 \nabla_{v,Jv}S(v,Jv)
\\&=\frac 1 2 \pi_2^*(D_{g_0}D_{g_0}^*S)-\frac 1 4 \nabla_{v,Jv}S(v,Jv)\approx-\frac 1 4 \nabla_{v,Jv}S(v,Jv).
\end{align*}
Now, using the fact that $J$ is parallel, we see that  $ \nabla_{v,Jv}S(v,Jv)=Y(\nabla S(Jv,v,Jv))$ is a co-boundary and we can thus ignore this term when computing the Hessian. Finally, we use \eqref{eq:V(S)} to obtain \eqref{eq:QCH^n}.

\textbf{Quotient of $\mathbb H \mathbb H^n$.}
The computation follows from the one for $\mathbb C \mathbb H^n$. Indeed, this time
$$ \sum_{j=1}^n \mu_j^{-1}S(Y_j,Y_j)=\sum_{j=1}^nS(Y_j,Y_j)+ \sum_{i=1}^3S(Y_{i+1},Y_{i+1})=\tr(S)+\pi_2^*(\sum_{i=1}^3S\circ J_i),$$
by \eqref{eq:vect}. This also shows that is term is independent of the choice of the base $(J_1,J_2,J_3)$ and thus globally defined. The last term $-\frac 1 2 \sum_{j=1}^n \mu_j\nabla_{Y,Y_j}^2S(Y,Y_j)$ is still a co-boundary by the exact same computation, using the fact that the $Y_i$ are parallel. We are left with understanding the middle term: 
\begin{align*}-\frac 1 4 \sum_{j=1}^n \mu_j\nabla_{Y_j,Y_j}^2S(Y,Y) &=-\frac 1 4 \sum_{j=1}^n \nabla_{Y_j,Y_j}^2S(Y,Y)+\frac 1 8 \sum_{j=1}^3\nabla_{Y_{j+1},Y_{j+1}}S(Y,Y)
\\&=\frac 1 4\pi_2^*(\nabla^*\nabla S)+\frac 1 8 \sum_{j=1}^3\nabla_{J_{j+1}Y,J_{j+1}Y}S(Y,Y).
\end{align*}
The last equality again follows from \eqref{eq:vect}. Indeed, there exists $(a_{i,j})_{1\leq i,j\leq 3}\in \mathbb R^9$ such that
$$Y_{j+1}=\sum_{i=1}a_{i,j}J_iY, \quad A=(a_{i,j})_{1\leq i,j\leq 3}\in \mathrm{O}_3(\mathbb R). $$
This means that 
\begin{align*}
 \sum_{j=1}^3\nabla_{Y_{j+1},Y_{j+1}}S(Y,Y)&= \sum_{j=1}^3\nabla_{\sum_{i=1}a_{i,j}J_iY,\sum_{i=1}a_{i,j}J_iY}S(Y,Y)
 \\&=\sum_{i,k,j}a_{i,j}a_{k,j}\nabla_{J_iY,J_kY}S(Y,Y)=\sum_i \nabla_{J_iY,J_iY}S(Y,Y).
\end{align*}
 We obtain, similarly to the complex hyperbolic case the fact that 
 $$\mathrm{Sym}\left(\sum_{j=1}^3\nabla_{J_{j+1}Y,J_{j+1}Y}S(Y,Y)\right)= -\pi_4^*\left(\sum_{i=1}^3D_{g_0}^{J_i}D_{g_0}^{J_i}S\right).$$
Note that the argument did not use the fact that $J$ is parallel. Combining everything gives \eqref{eq:QHH^n}.

\textbf{Quotient of $\mathbb O\mathbb H^n$.} This time the "trace" term in \eqref{eq:V(S)} gives
\begin{align*} \frac 1 2\sum_{j=1}^n \mu_j^{-1}S(Y_j,Y_j)&=\frac 1 2 \tr(S)+\frac 1 2 \sum_{i=1}^7S(Y_{i+1},Y_{i+1})
\\&= \frac 1 2 \tr(S)+\frac 16\left(\pi_2^*S-\tr(S)-\pi_2^*R^{\circ}(S) \right)
\\&=\frac 1 3 \tr(S)+\frac 16 \pi_2^*S-\frac 16 \pi_2^*R^{\circ}(S).
\end{align*}
where we used \eqref{eq:ww}. 
The second to last term is still a coboundary and the remaining one is computed using using the fact that the $Y_j$ are parallel
\begin{align*}-\frac 1 4 \sum_{j=1}^n \mu_j\nabla_{Y_j,Y_j}^2S(Y,Y) &=-\frac 1 4 \sum_{j=1}^n \nabla_{Y_j,Y_j}^2S(Y,Y)+\frac 1 8 \sum_{j=1}^7\nabla_{Y_{j+1},Y_{j+1}}S(Y,Y)
\end{align*}
For a $4$-tensor $T$, we extend the operator $R^{\circ}$ in the following way:
\begin{equation}
\label{eq:RR}
R^{\circ}(T)(X_1,X_2,X_3,X_4)=-\sum_{i=1}^n T(R(e_i,X_1)X_2,e_i,X_3,X_4). 
\end{equation}
Mimicking the computation of \eqref{eq:ww}, we obtain for $T=\nabla^2S$,
\begin{align*}3\sum_{j=1}^7\nabla_{Y_{j+1},Y_{j+1}}S(Y,Y)&=\pi_4^*(\nabla^2 S)-\pi_2^*(\tr \nabla^2 S)-\pi_4^*(R^{\circ}(\nabla^2 S))
\\&=\pi_4^*(D_{g_0}D_{g_0} S)+\pi_2^*(\nabla^*\nabla S)-\pi_4^*(\mathrm{Sym}(R^{\circ}(\nabla^2 S))). \end{align*}
The first term is a co-boundary which we can ignore in the computation of the Hessian. Summing the previous computations gives
$$V(x,v)=\pi_2^*\big(\frac 1 3 \tr(S)g_0+\frac 16 S-\frac 16 R^{\circ}(S)+\frac 7{24}\nabla^*\nabla S\big)-\frac 1{24}\pi_4^*(\mathrm{Sym}(R^{\circ}(\nabla^2 S)))$$
which is \eqref{eq:OH^n}, using the relation $\pi_{m+2}^*L(T)=\pi_m^*T$ for any symmetric $m$-tensor.
\end{proof}
\subsection{Stability estimate}
\label{stab}
In this subsection, we show that Theorem \ref{MainTheo} follows from the microlocal approach if $\Pi_{\mathrm{Ker}(D_{g_0}^*)}Q$ is \emph{solenoidal injective}. The solenoidal injectivity for real and complex hyperbolic metrics will be shown in Section \ref{sectionsinj}.
\begin{defi}
We say that an operator $V:C^{\infty}(M;S^2T^*M)\to C^{\infty}(M;S^mT^*M)$ is solenoidal injective if 
$$ \mathrm{Ker}(D_{g_0}^*)\cap \mathrm{Ker}(V)=\{0\}.$$
\end{defi}
We first prove that the averaging term in \eqref{eq:theoHess} is bounded by the volume term in \eqref{eq:stability}.
\begin{lemm}[Bounding the average term]
\label{averaging}
With the notation of Section \ref{secHessian}, there is a neighborhood $\mathcal U$ of $g_0$ in the $C^{5,\alpha}$-topology and $C_n>0$, such that if $g\in \mathcal U$ with $g-g_0\in \mathrm{Ker}(D_{g_0}^*)$, one has
$$\langle \pi_{4}^*Q(g-g_0), 1\rangle_{L^2(SM)}\leq C_n  |\mathrm{Vol}_g(M)-\mathrm{Vol}_{g_0}(M)|+C_n\|g-g_0\|_{C^{5,\alpha}}^2.$$
\end{lemm}
\begin{proof}
In the following proof, except if stated otherwise, the notation $\langle \cdot, \cdot \rangle$ will denote the scalar product induced on symmetric tensors.
If $g_0$ is a hyperbolic metric, then from \eqref{eq:QH^n}, one has
$$Q(S)=-\frac 1 2 S+\frac 1 4 \nabla^* \nabla S+\frac 12 (\mathrm{tr}_{g_0} S)g_0.$$
This means that using \eqref{eq:conformal}, one has
$$\langle \pi_2^*Q(S),1\rangle_{L^2(SM)}=\langle \pi_2^*Q(S),\pi_2^*g_0\rangle_{L^2(SM)}=C_n\langle Q(S),g_0\rangle_{L^2(S^2T^*M)}, $$
for some constant $C_n>0$.
Now, we have
\begin{align*}
\left\langle -\frac 1 2 S+\frac 1 4 \nabla^* \nabla S+\frac 12 (\mathrm{tr}_{g_0} S)g_0,g_0\right\rangle &=-\frac 1 2 \langle S,g_0\rangle +\frac 1 4\langle \nabla^*\nabla S, g_0\rangle+\frac 1 2\langle (\mathrm{tr}_{g_0} S)g_0,g_0\rangle
\\&=-\frac 1 2 \langle S,g_0\rangle+\frac n 2\langle \mathrm{tr}_{g_0} S,1\rangle
\\&=-\frac 1 2 \langle S,g_0\rangle+\frac n 2 \langle S,g_0\rangle=\frac {n-1} 2\langle S,g_0\rangle .
\end{align*}
Here, we used that $\nabla g_0=0$. Now, we follow the argument of \cite[Proposition 3.5]{GuKnLef} to obtain:
\begin{equation}
\label{eq:volume}
\mathrm{Vol}_g(M)-\mathrm{Vol}_{g_0}(M)=\frac 1 2 \langle g-g_0,g_0\rangle+O(\|g-g_0\|_{C^{5,\alpha}}^2),
\end{equation}
which gives the desired inequality.

Now, suppose that $g_0$ is complex hyperbolic, then \eqref{eq:QCH^n} gives 
$$Q(S)=\frac 1 4 L(\nabla^* \nabla S)+\frac 1 2 \big({-}L(S)+ (\mathrm{tr}_{g_0} S)L(g_0)+L(S\circ J)\big)+\frac 1 8 (D_{g_0}D_{g_0}(S\circ J))\circ J.$$
This time, we have using \eqref{eq:conformal}
$$\langle \pi_4^*Q(S),1\rangle_{L^2(SM)}=\langle \pi_4^*Q(S),\pi_4^*L(g_0)\rangle_{L^2(SM)}=\Lambda_4^n\langle Q(S),L(g_0)\rangle_{L^2(S^4T^*M)}.$$
We will write $Q(S)=L(T(S))-\tfrac 18 D_{g_0}^JD_{g_0}^J(S)$ for a symmetric $2$-tensor $T(S)$ defined in \eqref{eq:QCH^n}.
We note that for any symmetric $2$-tensor $T$, a direct computation yields
$$\langle L(T),L(g_0)\rangle=\langle T, \tr L(g_0)\rangle=\left(n+\frac {n+2}6\right)\langle T,g_0\rangle. $$
Here, we used the commutation relation between $\tr$ and its adjoint $L$ (see for instance \cite[Equation $(4.19)$ Chapter $4.1$]{GuMaz}) which gives
$$\tr L(g_0)=L(\tr(g_0))+\frac {n+2}6g_0=\left(n+\frac {n+2}6\right)g_0. $$
In particular, one has
\begin{align*}&\left\langle\frac 1 4 L(\nabla^* \nabla S)+\frac 1 2 \big({-}L(S)+ (\mathrm{tr}_{g_0} S)L(g_0)+L(S\circ J)\big), L(g_0)\right\rangle
\\&=C_n\left\langle -\frac 1 2 S+\frac 1 4 \nabla^* \nabla S+\frac 12 (\mathrm{tr}_{g_0} S)L(g_0)+\frac 12 S\circ J,g_0\right\rangle=C_n'\frac n 2 \langle S,g_0\rangle,
\end{align*}
where we used that $\langle S\circ J,g_0\rangle= \langle S, g_0\rangle $, did the same calculation as in the hyperbolic case and where $C_n, C_n'>0$ are some constants depending on $n$. In particular, the term $L(T(S))$ can be bounded just like in the first case.

We are left with bounding the term $ \langle (D_{g_0}D_{g_0}(S\circ J))\circ J, L(g_0)\rangle$. For this, we write
$$  \langle (D_{g_0}D_{g_0}(S\circ J))\circ J, L(g_0)\rangle= \langle D_{g_0}D_{g_0}(S\circ J), L(g_0\circ J)\rangle= \langle D_{g_0}(S\circ J), D_{g_0}^*L(g_0)\rangle,$$
we used the fact that $g_0\circ J=g_0$ as well as the fact that $L$ and $J$ commute. Now, we can use the commutation relation between $D_{g_0}^*$ and $L$ (see for instance \cite[Corollary 4.13]{GuMaz}) which gives, for some constant $C_n$,
$$ D_{g_0}^*L(g_0)=LD_{g_0}^*g_0+C_nD_{g_0}g_0=0.$$
This shows that this last term does not contribute to the average and concludes the proof in this case.

For the case of a quaternionic hyperbolic quotient, we see that the previous computation did not use the fact that $J$ was parallel. In particular, similar bounds for each $J_i$ hold and combining this with formula \eqref{eq:QHH^n} yields the result in this case.

For the octonionic case, the scalar products we are yet to bound are $\langle L(R^{\circ}(S)),L(g_0)\rangle$, $\langle \mathrm{Sym}(R^{\circ}(\nabla^2S)),L(g_0)\rangle $ as well as $\langle \mathrm{Sym}(\nabla^2S),L(g_0)\rangle$. For the first scalar product, we first note from \eqref{eq:Ricci} that $R^{\circ}$ is selfadjoint. Now,
$$
\langle L(R^{\circ}(S)),L(g_0)\rangle=C_n\langle S,R^{\circ}(g_0)\rangle,$$
for some constant $C_n>0$.
We use  \cite[Corollary 5.5]{Ruan} to get
\begin{equation}
\label{eq:Ruan}
R^\circ (g_0)(e_j,e_j) =-\sum_{i,j=1}^n g_0(R(e_i,e_j)e_j,e_i)=(-7\times 4-(16-8))=-36
\end{equation}
for an local orthonormal frame $(e_j)$. This means that $R^\circ (g_0)=-36g_0$. In particular, one gets
$$\langle L(R^{\circ}(S)),L(g_0)\rangle=C_n'\langle S, g_0\rangle,$$
for some $C_n'>0$. This term can be bounded using \eqref{eq:volume}. Next, 
\begin{align*}\langle \mathrm{Sym}(R^{\circ}(\nabla^2S)),L(g_0)\rangle_{L^2(S^2T^*M)}&=\langle \nabla^2 S,R^\circ(L(g_0))\rangle_{L^2(S^2T^*M)}.
\end{align*}
We now remark using \eqref{eq:Ruan} that $$\pi_4^*\big(R^\circ(L(g_0))\big)(x,v)=\pi_4^*\big(R^\circ(g_0\otimes g_0)\big)(x,v)=-36 \in \Omega_0.$$
From \eqref{eq:conformal}, this implies $R^\circ(L(g_0))=-36g_0$.
 This shows that this pairing vanishes and does not contribute to the average.
For the last term, we use the fact that $\mathrm{Sym}$ is an orthogonal projection to obtain
$$\langle \mathrm{Sym}(\nabla^2S),L(g_0)\rangle=\langle \nabla^2S,L(g_0)\rangle=\langle \nabla S, D_{g_0}^*L(g_0)\rangle =0 $$
and so this last term does not contribute to the average. 
\end{proof}
Note that up to here, using the computation of the Hessian \eqref{eq:theoHess}, Lemma \ref{averaging} and the coercive estimate \eqref{eq:coercive} gives
\begin{equation}
    \label{eq:first}
    \begin{split}
  \|\Pi_{\mathrm{Ker}(D_{g_0}^*)}Q(S)\|_{H^{-1/2}}^2   \leq C_n\langle d^2\Phi(g_0)S,S\rangle&+ C_n\big(|\mathrm{Vol}_g(M)-\mathrm{Vol}_{g_0}(M)|
  \\&+\|g-g_0\|_{C^{5,\alpha}}^2\big)^2.
  \end{split}
\end{equation}

Next, we show that the operator $\Pi_{\mathrm{Ker}(D_{g_0}^*)}Q$ is elliptic on solenoidal tensors.
\begin{prop}[Ellipticity]
\label{elliptic}
With the notations introduced in the section, the operator $\Pi_{\mathrm{Ker}(D_{g_0}^*)}Q$ is elliptic on solenoidal tensors in the sense that there exists $C>0$, such that for any $S\in \mathrm{Ker}(D_{g_0}^*) $, one has
$$C\|S\|_{H^{3/2}}^2\leq \|\Pi_{\mathrm{Ker}(D_{g_0}^*)}Q S\|^2_{H^{-1/2}}+\|\Pi_{\mathrm{ker}(\Pi_{\mathrm{Ker}(D_{g_0}^*)}Q)} S\|^2_{H^{3/2}} $$
where $\Pi_{\mathrm{ker}(\Pi_{\mathrm{Ker}(D_{g_0}^*)})}$ is the orthogonal projection on the kernel and $C>0$.
\end{prop}
\begin{proof}
We will compute the principal symbol of $Q$ in each case. Note that by \cite[Lemma 14.1.11]{Lef}, $\Pi_{\mathrm{Ker}(D_{g_0}^*)}$ is a pseudo-differential operator of order $0$ with principal symbol equal to $\sigma_{\Pi_{\mathrm{Ker}(D_{g_0}^*)}}(x,\xi)=\pi_{\mathrm{Ker}(\iota_{\xi})}$ where $\iota_{\xi}$ denotes the contraction with the tangent vector $\xi^{\sharp}$ and $\pi_{\mathrm{Ker}(\iota_{\xi})}$ is the projection onto ${\mathrm{Ker}(\iota_{\xi})}$. We will thus obtain the principal symbol of $\Pi_{\mathrm{Ker}(D_{g_0}^*)}Q$ by composing with the principal symbol of $Q$.

If $g_0$ is a hyperbolic metric, from \eqref{eq:QH^n} we immediately see that
$$\sigma_Q(x,\xi)=\frac 14 \|\xi\|^2 ,\quad  \sigma_{\Pi_{\mathrm{Ker}(D_{g_0}^*)}Q}(x,\xi)=\frac 14\pi_{\mathrm{ker}(\iota_{\xi})}\|\xi\|^2, $$
which proves that $\Pi_{\mathrm{Ker}(D_{g_0}^*)}Q$ is elliptic on solenoidal tensors, i.e., its principal symbol satisfies $\sigma_{\Pi_{\mathrm{Ker}(D_{g_0}^*)}Q}(x,\xi)\geq c\pi_{\mathrm{ker}(\iota_{\xi})}\|\xi\|^2$ for some $c>0$ and $\|\xi\|\gg 1$.

If $g_0$ is either complex hyperbolic or quaternionic hyperbolic, then we prove that $\Pi_{\mathrm{Ker}(D_{g_0}^*)}Q$ is of gradient type on $\mathrm{Ker}(D_{g_0}^*)$, i.e., its principal symbol is injective on $\mathrm{Ker}(D_{g_0}^*)$ at all points. Using \eqref{eq:QCH^n} or \eqref{eq:QHH^n}, we see that
$$\sigma_{Q}=\frac 14 \sigma_{L(\nabla^*\nabla)}+\frac 18 \sum_{i=1}^m\sigma_{(D_{g_0}D_{g_0}(S\circ J_i))\circ J_i}, $$
with $m=1$ for a quotient of a complex hyperbolic space and $m=3$ for a quotient of the quaternionic hyperbolic space.
The principal symbol of $\nabla^*\nabla$ is $\|\xi\|^2$ and using \cite[Lemma 14.1.11]{Lef}, the principal symbol of $D_{g_0}$ is $ij_{\xi}:=i\mathrm{Sym}(\xi^{\sharp}\cdot)$ where $\xi^{\sharp}$ is the image by the musical isomorphism of the co-vector $\xi$. In total, this gives
$$\sigma_{\Pi_{\mathrm{Ker}(D_{g_0}^*)}Q}(x,\xi)S=\frac 1 4\|\xi\|^2\pi_{\mathrm{ker}(\iota_\xi)}L(S)-\frac 18 \sum_{i=1}^m \pi_{\mathrm{ker}(\iota_\xi)}j_{J_i\xi} j_{J_i\xi} S. $$
Consider $(x,\xi) \in T(SM)$ with $\xi \neq 0$, then suppose that $S\in S^2T^*_xM$ is such that $\sigma_{\Pi_{\mathrm{Ker}(D_{g_0}^*)}Q}(x,\xi)S=0$. First, we notice that $L$ and $D_{g_0}^*$ commute which gives $L$ and $\pi_{\mathrm{ker}(\iota_\xi)}$ commute. Moreover, we see that $\iota_\xi j_{J_i\xi}=0$ because $g(J_i\xi^{\sharp}, \xi^{\sharp})=0$. In particular, one has 
$$0=\sigma_{\Pi_{\mathrm{Ker}(D_{g_0}^*)}Q}(x,\xi)S=\frac 1 4\|\xi\|^2L(S)-\frac 18 \sum_{i=1}^m j_{J_i\xi} j_{J_i\xi} S. $$
This implies $\pi_4^*\sigma_{\Pi_{\mathrm{Ker}(D_{g_0}^*)}Q}(x,\xi)S=0$, in other words
$$\frac 14 \|\xi\|^2\pi_2^*S-\frac 18 \sum_{i=1}^m\big(\xi(J_iv)\big)^2\pi_2^*S=\frac 14 \underbrace{\left(\|\xi\|^2-\frac 12  \sum_{i=1}^m\big(\xi(J_iv)\big)^2\right)}_{\geq \tfrac 12 \|\xi\|^2}\pi_2^*S=0. $$
But then $\pi_2^*S=0$ and thus $S=0$. In particular, because the principal symbol is homogenous in $\xi$, this shows that $\Pi_{\mathrm{Ker}(D_{g_0}^*)}Q$ is of gradient type on $\mathrm{Ker}(D_{g_0}^*)$. 

Finally, if $g_0$ is an octonionic hyperbolic metric, we obtain from \eqref{eq:OH^n},
$$24\sigma_{Q}=7 \sigma_{L(\nabla^*\nabla)}- \sigma_{\mathrm{Sym}(R^{\circ}(\nabla^*\nabla))}. $$
The principal symbol of the second term can be computed as follows. Fix $(x,\xi)\in SM$ with $\xi\neq 0$ and choose a phase function $\varphi \in C^{\infty}(M)$ such that $d\varphi(x)=\xi$. Then
$$\sigma_{\mathrm{Sym}(R^{\circ}(\nabla^*\nabla))}(x,\xi)S=\lim_{h\to0}{h^2}e^{-\frac{i\varphi}h}\mathrm{Sym}(R^{\circ}(\nabla^*\nabla))(e^{\frac{i\varphi}{h}}S)=-\mathrm{Sym}(R^{\circ}(\xi(\cdot)^2S)), $$ see \cite[Lemma 5.1.15]{Lef}. Fix an orthonormal basis $(e_i)_{i=1}^n$ such that $v_1=v$ and $v_1,\ldots, v_8$ span $\mathrm{Cay}(v)$. Using \eqref{eq:RR}, 
$$-\pi_2^*\mathrm{Sym}(R^{\circ}(\xi(\cdot)^2S))(x,v)=\sum_{i=1}^n\xi(R(e_i,v)v)\xi(e_i)S_x(v,v) $$
which then gives, similarly to previous cases,
$$24\pi_4^*(\sigma_{Q}S)(x,v)\geq \underbrace{(7\|\xi\|^2-4\sum_{i=1}^{n}\xi(e_i)^2)}_{\geq  3 \|\xi\|^2}\pi_2^*S. $$
From this estimate, we conclude that the principal symbol is injective.

In all the previous cases, this means that $(\Pi_{\mathrm{Ker}(D_{g_0}^*)}Q)^*\Pi_{\mathrm{Ker}(D_{g_0}^*)}Q$ is elliptic on solenoidal tensors and thus there exists a sharp parametrix (see \cite[Theorem 6.1.3]{Lef}), i.e., there exists $A\in \Psi^{-2}$ a pseudo-differential operator of order $-2$ such that
$$A\Pi_{\mathrm{Ker}(D_{g_0}^*)}Q=\Pi_{\mathrm{Ker}(D_{g_0}^*)}-\Pi_{\mathrm{ker}(\Pi_{\mathrm{Ker}(D_{g_0}^*)}Q)}. $$
This finally gives, for a divergence-free tensor $S$,
\begin{align*}\|S\|_{H^{3/2}}^2&\leq 2 \|A\Pi_{\mathrm{Ker}(D_{g_0}^*)}Q S\|^2_{H^{3/2}}+2\|\Pi_{\mathrm{ker}(\Pi_{\mathrm{Ker}(D_{g_0}^*)}Q)} S\|^2_{H^{3/2}}
\\&\leq C\|\Pi_{\mathrm{Ker}(D_{g_0}^*)}Q S\|^2_{H^{-1/2}}+C\|\Pi_{\mathrm{ker}(\Pi_{\mathrm{Ker}(D_{g_0}^*)}Q)} S\|^2_{H^{3/2}},
\end{align*}
for some $C>0$.
This concludes the proof.
\end{proof}
We are now ready to show that solenoidal injectivity implies our main theorem. 
\begin{prop}
\label{propprop}
Under the hypothesis of Theorem \ref{MainTheo}, if $\Pi_{\mathrm{Ker}(D_{g_0}^*)}Q$ is solenoidal injective $($where $Q$ is defined in Theorem \ref{theoHess}$)$ then one has \eqref{eq:stability}.
\end{prop}

\begin{proof}
Recalling \eqref{eq:theoHess} and \eqref{eq:Pi2}, we have
$$
\langle d^2\Phi(g_0)S,S\rangle_{L^2(S^2T^*M)} =\langle\Pi_{m(g_0)}^{g_0}Q(S),Q(S)\rangle - \langle \pi_{m(g_0)}^*Q(S), 1\rangle^2_{L^2(SM)}.
$$
Let $g\in \mathcal U$. We can use a Taylor expansion as well as the slice Lemma \ref{slice lemma} to obtain for any $\alpha\in(0,1)$ a diffeomorphism $\phi$ such that
\begin{align*}\Phi(g)=\Phi(\phi^*g)=C_n(\langle\Pi_{m(g_0)}^{g_0}Q(\phi^*g-g_0),Q(\phi^*g-g_0)\rangle &- \langle \pi_{m(g_0)}^*Q(\phi^*g-g_0), 1\rangle^2)
\\&+O(\|g-g_0\|_{C^{5,\alpha}}^3).
\end{align*}
The central point is that $S:=\phi^*g-g_0$ is solenoidal so we can use the coercive estimate \eqref{eq:coercive} and Lemma \ref{averaging}
$$\|\Pi_{\mathrm{Ker}(D_{g_0}^*)}Q(\phi^*g-g_0)\|_{H^{-1/2}}^2\leq C'_n\Phi(g)+C_n'\|\phi^*g-g_0\|_{C^{5,\alpha}}^3+C_n  (\mathrm{Vol}_g(M)-\mathrm{Vol}_{g_0}(M))^2. $$
To conclude we use Proposition \ref{elliptic} and the solenoidal injectivity to get
$$C\|S\|_{H^{3/2}}^2\leq \|\Pi_{\mathrm{Ker}(D_{g_0}^*)}Q S\|^2_{H^{-1/2}}+\|\Pi_{\mathrm{ker}(\Pi_{\mathrm{Ker}(D_{g_0}^*)}Q)} S\|^2_{H^{3/2}}= \|\Pi_{\mathrm{Ker}(D_{g_0}^*)}Q S\|^2_{H^{-1/2}}.$$
Plugging this into the last equation gives
$$ \|\phi^*g-g_0\|_{H^{3/2}}^2\leq C'_n\Phi(g)+C_n'\|\phi^*g-g_0\|_{C^{5,\alpha}}^3+C_n  (\mathrm{Vol}_g(M)-\mathrm{Vol}_{g_0}(M))^2.$$
We can use Sobolev embedding and an interpolation argument, using the fact that $\phi^*g-g_0$ is small in $C^N$ norm to absorb the term $C_n'\|\phi^*g-g_0\|_{C^{5,\alpha}}^3$ in the left hand side. More precisely, for any $\beta>\alpha$, 
$$\|\phi^*g-g_0\|_{C^{5,\alpha}}^3\leq c_{g_0}\|\phi^*g-g_0\|_{H^{n/25+\beta}}^3\leq c'_{g_0}\|\phi^*g-g_0\|_{H^{3/2}}^2\|\phi^*g-g_0\|_{C^N}, $$
for any $N>\tfrac 3 2 n+6+6\beta$. This last estimate gives \eqref{eq:stability}.
\end{proof}

\section{Solenoidal injectivity of the operator $\Pi_{\mathrm{Ker}(D_{g_0}^*)}Q.$}
\label{sectionsinj}
In this section, we prove that  $\Pi_{\mathrm{Ker}(D_{g_0}^*)}Q$ is solenoidal injective. This uses the geometry of the different types of locally symmetric spaces and relies on a Weitzenböck or Bochner formulas on symmetric tensors.
\begin{prop}[Solenoidal injectivity]
\label{sinj}
With the notations of the previous section, the operator $\Pi_{\mathrm{Ker}(D_{g_0}^*)}Q$ is solenoidal injective for real hyperbolic and complex hyperbolic metrics.
\end{prop}
We make the following conjecture concerning the remaining cases, we insist on the fact that this last step seems to be purely geometric and dynamical but does not seem to rely on microlocal analysis.
\begin{conj}
\label{conj}
The operator $\Pi_{\mathrm{Ker}(D_{g_0}^*)}Q$ is solenoidal injective for any locally symmetric metric.
\end{conj}
Together with Proposition \ref{propprop}, this would prove Katok's entropy conjecture near any locally symmetric metric.
\subsection{Quotients of $\mathbb R \mathbb H^n$.}
Recall that in this case, $Q$ is given by
$$
Q(S)=-\frac 1 2 S+\frac 1 4 \nabla^* \nabla S+\frac 12 (\mathrm{tr}_{g_0} S)g_0.$$
Suppose that there exists $S\in \mathrm{Ker}(D_{g_0}^*)$ such that $\Pi_{\mathrm{Ker}(D_{g_0}^*)}Q S=0$. In particular, one has
$$\langle \Pi_{\mathrm{Ker}(D_{g_0}^*)}Q (S),S\rangle_{L^2(S^2T^*M)} = \langle Q(S),S\rangle_{L^2(S^2T^*M)} =0. $$
We first decompose $S=S_0+L(h)$ with $S_0\in C^{\infty}(M;S^2_0T^*M)$ trace-free and $h\in C^{\infty}(M)$. As the rough Laplacian commutes with the trace, we see that $Q$ preserves the decomposition, in the sense that 
$$\langle Q(S),S\rangle_{L^2(S^2T^*M)} = \langle Q(S_0),S_0\rangle_{L^2(S^2T^*M)} +\langle Q(L(h)),L(h)\rangle_{L^2(S^2T^*M)} . $$
To estimate the first term,  we will use the Weitzenböck identity \eqref{eq:Weit1}:
\begin{align*}
\langle Q(S_0),S_0\rangle_{L^2(S^2T^*M)} &=-\frac 1 2\|S_0\|_{L^2(S^2T^*M)} ^2+\frac 1 4 \langle\nabla^*\nabla S_0,S_0\rangle_{L^2(S^2T^*M)} 
\\&+\frac 1 2\langle  (\mathrm{tr}_{g_0} S_0)g_0 ,S_0\rangle_{L^2(S^2T^*M)}  \geq \frac 1 4(n-2)\|S_0\|_{L^2(S^2T^*M)} ^2.
\end{align*}
Now, the second part is easier to bound as
$$\langle Q(hg_0),hg_0\rangle_{L^2(S^2T^*M)}= \frac n 4\langle h,\Delta h+2(n-1)h\rangle_{L^2(M)}\geq \frac{n(n-1)}{2}\|h\|_{L^2(M)}^2.$$
In total this gives,
$$\frac 1 4(n-2)\|S_0\|_{L^2(S^2T^*M)}^2+\frac n 2 (n-1)\|h\|_{L^2(M)}^2\leq \langle Q(S),S\rangle_{L^2(S^2T^*M)} =0.  $$
This implies $S=0$ if $n\geq 3$ and thus the solenoidal injectivity holds.

\subsection{Quotients of $ \mathbb C \mathbb H^{n/2}$.}
The real dimension $n$ is even in this case. Note however that any quotient of $\mathbb C\mathbb H^1$ is actually isometric to a real-hyperbolic case, see \cite[Proposition 10.12]{BriHa}. In the following, we will thus suppose without loss of generality that $n\geq 4$.

Suppose again for a  contradiction that  there exists $S\in \mathrm{Ker}(D_{g_0}^*)$ such that $Q(S)=D_{g_0}\tilde p$ for $\tilde p\in C^{\infty}(M;S^3T^*M)$. Note that this time, $Q(S)$ is a $4$-tensor and thus we cannot make sense of the scalar product $\langle Q(S),S\rangle$ anymore.

\textbf{Strategy for the solenoidal injectivity in the complex case.}
The idea is to first consider $p\in C^{\infty}(M;S_0^3T^*M)$ the trace-free part of $\tilde p$ and project $Q(S)=D_{g_0}\tilde p$ onto the trace-free component. In particular, using the Pestov identity \eqref{eq:X_+inj} will provide an expression of $p$ in terms of $S$. Now, this means that using the commutativity relation $[D_{g_0},L]=0$ :
$$Q(S)-D_{g_0}p=D_{g_0}(\tilde p- p)=D_{g_0}(L(q))=L(D_{g_0}q), $$
where $q\in C^{\infty}(M;T^*M)$ is a $1$-tensor. Since $p$ is a function of $S$, this means that
$$Q(S)-D_{g_0}p=L(W(S))=L(D_{g_0}q), $$
for some (explicit) operator $W: C^{\infty}(M;S^2T^*M)\to C^{\infty}(M;S^2T^*M)$.
The injectivity of $L$ gives $W(S)=D_{g_0}q$ and we have reduced ourselves to an equation on symmetric $2$-tensors. We can now consider the null-scalar product $\langle W(S),S\rangle =0$. The contradiction will then come from the coersive estimate $\langle W(S),S\rangle\geq C_n \|S\|^2$ for some constant $C_n>0$ as in the real hyperbolic case.

Recall from \eqref{eq:QCH^n} that the operator $Q$ is given by
$$
Q(S)=\frac 1 4 L(\nabla^* \nabla S)+\frac 1 2 \big({-}L(S)+ (\mathrm{tr}_{g_0} S)L(g_0)+L(S\circ J)\big)-\frac 1 8 D_{g_0}^JD_{g_0}^JS.$$
Let  $\tilde p\in C^{\infty}(M;S^3T^*M)$ such that $Q(S)=D_{g_0}\tilde p$ and let $p\in C^{\infty}(M;S_0^3T^*M)$ be the trace-free part of $\tilde p$. We will keep the notation $p$ to denote the pullback $\pi_3^*p\in \Omega_3$. By Lemma \ref{lemmV}, we know that $p=p_3+p_1+p_{-1}+p_{-3}$ with $p_\lambda\in E_3^\lambda$ and $p_{-\lambda}=\bar{p_\lambda}$. On the other hand, we will decompose $\pi_2^*S=f+h$ with $f\in \Omega_2$ and $h\in \Omega_0$. We can further decompose $f=f_2+f_0+f_{-2}$ with $f_j \in E^j_2$ in a similar manner to $p$. 

\textbf{Solving for the trace-free part of $\tilde p$.} We will now use the theory of raising/lowering operators defined in Subsection \ref{CH}.

Using \eqref{eq:XandD} and \eqref{eq:H=JXJ}, we obtain
$\pi_4^*(D_{g_0}^JD_{g_0}^JS)=-H^2\pi_2^*S. $ As a consequence of the mapping properties of $X$ recalled in Subsubsection \ref{X+X-}, the only contribution for $\Omega_4$ are $X_+p$ for $\pi_4^*D_{g_0}\tilde p$ and $-\tfrac 18 H_+H_+f$ for $\pi_4^*Q(S).$

Next, using the notations of Lemma \ref{lemmV}, the space $\Omega_4$ further splits into 
$$\Omega_4=E_4^4\oplus E_4^2\oplus E_4^0\oplus E_4^{-2}\oplus E_4^{-4}. $$
In particular, projecting the equation $Q(S)=D_{g_0}\tilde p$ onto $E_4^4,E^2_4$ and $E^0_4$ yields, using the mapping properties proved in Lemma \ref{lemmV},
\begin{equation}
\begin{cases}
8\eta_+^+p_3&=-\eta_+^+\eta_+^+f_2
\\
8(\eta_+^-p_3+\eta_+^+p_1)&=(\eta_+^+\eta_+^-+\eta_+^-\eta_+^+)f_2-\eta_+^+\eta_+^+f_0
\\
8(\eta_+^-p_1+\eta_+^+p_{-1})&=(\eta_+^+\eta_+^-+\eta_+^-\eta_+^+)f_0-\eta_+^-\eta_+^-f_2-\eta_+^+\eta_+^+f_{-2}.
\end{cases}
\end{equation}
Using Lemma \ref{lemmV} and the injectivity of $\eta_+^+$ on $E_4^4$, the first line gives $8p_3=-\eta_+^+f_2$. Re-injecting in the second line, using the commutativity of $\eta_+^-$ and $\eta_+^+$ (see \eqref{eq:comm3}) and using the injectivity of $\eta_+^+$ again gives $8p_1=3\eta_+^-f_2-\eta_+^+f_0$. In total, we thus have
\begin{align*}8p&=-\eta_+^+f_2+3\eta_+^-f_2-\eta_+^+f_0+3\eta_+^+f_{-2}-\eta_+^-f_0-\eta_+^-f_{-2}
\\&=-X_+f+4(\eta_+^-f_{2}+\eta_+^+f_{-2})=-X_+f+2X_+(f_{2}+f_{-2})-2iH_+(f_2-f_{-2})
\\&=-X_+f-\frac 1 2 X_+V^2f -H_+Vf=-X_+\big(\mathrm{Id}+\frac 1 2 V^2\big)f-H_+ Vf
\\&=-X_+Jf-H_+ Vf.
\end{align*}
In the previous computations, we have used \eqref{eq:JandV}.

We can now go back to the problem. Write $\tilde p=p+L(q)$ with $p$ the trace-free part of $p$ and $q\in C^{\infty}(M;T^*M)$. The previous computation shows that $p$ is given by
\begin{equation}
\label{eq:p}
\pi_3^*p=-\frac 1 8 (X_+Jf+H_+ Vf).
\end{equation}

\textbf{Computing $Q(S)-D_{g_0}p$ as a function of $S$.}
In particular, this means that one has, using $[D_{g_0},L]=0$, for some one-form $q$, 
$$Q(S)-D_{g_0}p=D_{g_0}(\tilde p- p)=D_{g_0}(L(q))=L(D_{g_0}q). $$
But there exists $W(S)\in C^{\infty}(M;S^2T^*M)$ such that $Q(S)-D_{g_0}p=L(W(S))$ and the injectivity of $L$ implies $W(S)=D_{g_0}q$. We now compute the tensor $W(S)$, using the fact that $H_+H_+f=8X_+p$
\begin{equation}
\label{eq:V}
\begin{split}
V(f,h):&=\pi_4^*(-D_{g_0}^JD_{g_0}^JS-8D_{g_0}p)=H^2\pi_2^*S-8X\pi_3^*p
\\&=(H_++H_-)^2(f+h)+(X_++X_-)(X_+Jf+H_+ Vf)
\\&=\big((H_+H_-+H_-H_+)f+H_+H_+h+X_-X_+Jf+X_-H_+ Vf\big)\in \Omega_2
\\&+\big(H_-H_-f+H_-H_+h\big)\in \Omega_0.
\end{split}
\end{equation}
There exists a unique $Q_0(S)\in C^{\infty}(M;S^2T^*M)$ such that $\pi_2^*Q_0(S)=V(f,h)$ and
\begin{equation}
\label{eq:W(S)}
W(S)=\frac 1 4\nabla^* \nabla S+\frac 1 2 (-S+ (\mathrm{tr}_{g_0} S)g_0+S\circ J)+\frac 1 8Q_0(S).
\end{equation}
Since $W(S)=D_{g_0}q$ and $S$ is divergence-free, one has $\langle W(S), S\rangle_{L^2(S^2T^*M)} =0$.

\textbf{Bounding $\langle W(S),S\rangle$ from below.}
The proof of Proposition \ref{sinj} reduces to the following estimate.
\begin{prop}[Coercive estimate for complex hyperbolic quotients]
\label{propcoercive}
Let $(M^n,g_0)$ be a compact quotient of the complex hyperbolic space. There exists $C_n>0$ such that for any $S\in C^{\infty}(M;S^2T^*M)\cap \mathrm{Ker}(D_{g_0}^*)$, one has
\begin{equation}
\label{eq:final}
 \langle W(S),S\rangle_{L^2(S^2T^*M)} \geq C_n(\|S_0\|_{L^2(S^2T^*M)}^2+\|h\|_{L^2(M)}^2),
\end{equation}
where $S=S_0+hg_0$ with $S_0\in C^{\infty}(M;S_0^2T^*M)$ and $h\in C^{\infty}(M)$.
\end{prop}
\begin{proof}
We will start by considering the first term  appearing in \eqref{eq:W(S)}:
\begin{equation}
\label{eq:Q_1}
Q_1(S):=\frac 1 4\nabla^* \nabla S+\frac 1 2 ({-}S+ (\mathrm{tr}_{g_0} S)g_0+S\circ J).
\end{equation}
By definition, one has $W(S)=Q_1(S)+\tfrac 18 Q_0(S)$ and we first investigate $\langle Q_1(S),S\rangle$. We remark that since $J$ and $\nabla^*\nabla$ commute with the trace, $Q_1$ preserves the decomposition $S=S_0+hg_0$ where $S_0$ is the trace-free part of $S$ and $h\in C^{\infty}(M)$. More explicitly
$$\langle Q_1(S),S\rangle_{L^2(S^2T^*M)}=\langle Q_1(S_0),S_0\rangle_{L^2(S^2T^*M)}+\langle Q_1(hg_0),hg_0\rangle_{L^2(S^2T^*M)}. $$
We first compute the second term, using $(hg_0)\circ J=hg_0$,
\begin{equation}
\label{eq:h}
\langle Q_1(hg_0),hg_0\rangle_{L^2(S^2T^*M)}=\frac n 4 \langle \Delta h, h\rangle_{L^2(M)}+\frac{n^2}{2}\|h\|^2_{L^2(M)}. 
\end{equation}
Now the trace-free part contribution is
\begin{equation}
\label{eq:S_0}
\langle Q_1(S_0),S_0\rangle=\frac 1 4 \langle \nabla^*\nabla S_0, S_0\rangle-\frac{1}{2}\|S_0\|^2+\frac 1 2 \langle S_0\circ J,S_0\rangle. 
\end{equation}
Summing \eqref{eq:S_0} and \eqref{eq:h}, we obtain
\begin{equation}
\label{eq:Q_1eq}
\langle Q_1(S),S\rangle=\frac 1 4 \langle \nabla^*\nabla S_0, S_0\rangle-\frac{1}{2}\|S_0\|^2+\frac 1 2 \langle S_0\circ J,S_0\rangle+\frac n 4 \langle \Delta h, h\rangle+\frac{n^2}{2}\|h\|^2. 
\end{equation}
Next, we study the contribution of $Q_0(S)$. We prove the following lemma:
\begin{lemm}
One has the following bound:
\begin{equation}
\label{eq:lower}
\langle Q_0 S,S\rangle_{L^2(S^2T^*M)}\geq -\frac{97}{96}\|\nabla_{g_0}S_0\|_{L^2(S^2T^*M)}^2-\frac{n+5}{4}\|\nabla h\|_{L^2(T^*M)}^2.
\end{equation}
\end{lemm}
\begin{proof}
Recall that $\pi_2^*S=f+h$ where $f\in \Omega_2$ and $h\in \Omega_0$. Using \eqref{eq:conformal}, one sees that
\begin{equation}
\label{eq:utile}
\langle Q_0 S,S\rangle_{L^2(S^2T^*M)}=\Lambda_2^n\langle V(f,h),f\rangle_{L^2(SM)} +\Lambda_0^n\langle V(f,h),h\rangle_{L^2(SM)}.   
\end{equation}
This allows us to do the calculations using spherical harmonics and in the following computations, $\langle\cdot, \cdot\rangle$ will denote the scalar product on $L^2(SM)$.
We compute the first term using \eqref{eq:V}, the fact that $X_+^*=-X_-$ and $H_+^*=-H_-$,
\begin{equation}
\label{eq:V(f,h)f}
\begin{split}
\langle V(f,h),f\rangle&=\langle (H_+H_-+H_-H_+)f+H_+H_+h+X_-X_+Jf+X_-H_+ Vf,f\rangle
\\&=-\|H_+f\|^2-\langle X_+Jf+ H_+Vf,X_+f\rangle-\|H_-f\|^2-\langle H_+h,H_-f\rangle
\end{split}
\end{equation}
We use the fact that $J$ is an isometry and \eqref{eq:H=JXJ} to obtain 
$$ \|H_+f\|^2=\langle JX_+Jf,JX_+Jf\rangle= \| X_+Jf\|^2.$$
Now, we observe that using \eqref{eq:XandH}, \eqref{eq:JandV} and decomposing into eigenspaces of $V$ gives:
$$
\begin{cases}
X_+f=\eta_+^+f_2+(\eta_+^+f_0+\eta_+^-f_2)+(\eta_+^-f_0+\eta_+^+f_{-2})+\eta_+^-f_{-2}
\\
X_+Jf=-\eta_+^+f_2+(\eta_+^+f_0-\eta_+^-f_2)+(\eta_+^-f_0-\eta_+^+f_{-2})-\eta_+^-f_{-2}
\\
H_+Vf=-2\eta_+^+f_2+2\eta_+^-f_2+2\eta_+^+f_{-2}-2\eta_+^-f_{-2}.
\end{cases} $$
This implies, using the fact that the eigenspaces of $V$ are orthogonal, that we have
\begin{equation}
\label{eq:celle}
\begin{split}\langle X_+Jf+H_+Vf,X_+f\rangle=&-3\| \eta_+^+f_2\|^2+\|\eta_+^+f_0+\eta_+^-f_2\|^2+\|\eta_+^-f_0+\eta_+^+f_{-2}\|^2\\&-3\|\eta_+^-f_{-2}\|^2
=\|X_+f\|^2-4(\| \eta_+^+f_2\|^2+\|\eta_+^-f_{-2}\|^2).
\end{split}
\end{equation}
To bound the term $\|H_-f\|^2$ we will use the fact that $S$ is divergence-free and \eqref{eq:divfree}.
$$
X_-f=\eta_-^+f_2+(\eta_-^-f_2+\eta_-^+f_0)+(\eta_-^-f_0+\eta_-^+f_{-2}) +\eta_-^-f_{-2}.
$$
But using \eqref{eq:divfree} with \eqref{eq:XandH} and projecting onto eigenspaces of $V$ gives 
\begin{equation}
\label{eq:zero}
D_{g_0}^*S=0\ \Rightarrow \ \eta_-^+f_2=\eta_-^-f_{-2}=0.
\end{equation}
We notice that this implies that $X_-f-X_-Jf=2(\eta_-^-f_2+\eta_-^+f_{-2})$. In particular, 
\begin{equation}
\label{eq:22}
\begin{split}\|H_-f\|^2=\|X_-Jf\|^2&\leq 2\|X_-f\|^2+2\|X_-f-X_-Jf\|^2
\\&=2\|X_-f\|^2+ 8(\| \eta_-^-f_2\|^2+\|\eta_-^+f_{-2}\|^2).
\end{split}
\end{equation}
Using Cauchy-Schwartz's inequality, we can bound the remaining term:
\begin{equation}
\label{eq:23}
\begin{split} -\langle H_+h,H_-f\rangle&\geq -\frac 1 2\|H_+h\|^2-\frac 12 \|H_-f\|^2
\\&\geq -\frac 1 2\|H_+h\|^2-\|X_-f\|^2- 4(\| \eta_-^-f_2\|^2+\|\eta_-^+f_{-2}\|^2). 
\end{split}
\end{equation}
Plugging \eqref{eq:celle} and \eqref{eq:23} into \eqref{eq:V(f,h)f} gives
\begin{align*}\langle V(f,h),f\rangle\geq &-(\|Hf\|^2+\|X_+f\|^2-4(\| \eta_+^+f_2\|^2+\|\eta_+^-f_{-2}\|^2)+\frac 1 2 \|H_+h\|^2)
\\&-(\|X_-f\|^2+4(\| \eta_-^-f_2\|^2+\|\eta_-^+f_{-2}\|^2)).
\end{align*}
Now, we can use \eqref{eq:NewPestov} together with \eqref{eq:Pestov},\eqref{eq:PestovH} to get $ \|\eta_-^-f_2\|^2\leq \|\eta_+^+f_2\|^2$, as well as $ \|\eta_-^+f_{-2}\|^2\leq \|\eta_+^-f_{-2}\|^2$. The second inequality is deduced from the first one by complex conjugation. In total, we have obtained
\begin{equation}
\label{eq:Vf}
\langle V(f,h),f\rangle\geq -(\|Xf\|^2+\|Hf\|^2)-\frac 1 2 \|H_+h\|^2.
\end{equation}
The second term in \eqref{eq:utile} is bounded using Cauchy-Schwartz's inequality and \eqref{eq:PestovH}. Indeed, from \eqref{eq:V} recall that the $\Omega_0$-component of $V(f,h)$ is $H_-H_-f+H_-H_+h$.\begin{equation}
\begin{split}
\label{eq:2}
\langle V(f,h),h\rangle&\geq -\|H_+h\|^2-\langle 2H_-f,\tfrac 1 2H_+h\rangle \geq -3\|H_+h\|^2-\frac 1 8 \|H_+f\|^2
\\&\geq-3\|H_+h\|^2-\frac 1 8(\|Xf\|^2+\|Hf\|^2).
\end{split}
\end{equation}

We plug the lower bounds \eqref{eq:Vf} and \eqref{eq:2} into \eqref{eq:utile}. We then use the bounds of Lemma \ref{lemm1}, more precisely, we use the first bound on $\|H_+h\|^2_{L^2(SM)}$ and the second one on $\|Xf\|_{L^2(SM)}^2+\|Hf\|_{L^2(SM)}^2$ and $\|\nabla h\|_{L^2(T^*M)}$: 
\begin{align*}
\langle Q_0 S,S\rangle_{L^2(S^2T^*M)}&=\Lambda_2^n\langle V(f,h),f\rangle_{L^2(SM)} +\Lambda_0^n\langle V(f,h),h\rangle_{L^2(SM)}
\\& \geq -\Lambda_2^n(\Lambda_2^n)^{-1}\|\nabla_{g_0}S_0\|_{L^2(S^3T^*M)}^2-\frac{1}2\Lambda_2^n(\Lambda_1^n)^{-1}\|\nabla h\|_{L^2(T^*M)}^2
\\&-\frac 1 8\Lambda_0^n(\Lambda_2^n)^{-1}\|\nabla_{g_0}S_0\|_{L^2(S^3T^*M)}^2
- 3\Lambda_0^n(\Lambda_1^n)^{-1}\|\nabla h\|_{L^2(T^*M)}^2
\\&\geq -\left(1+\frac{1}{8n(n/2+1)} \right)\|\nabla_{g_0}S_0\|_{L^2(S^2T^*M)}^2
\\&-\left(\frac{n+2}{4}+\frac{3}{n}\right)\|\nabla h\|_{L^2(T^*M)}^2,\end{align*}
where we used that $\Lambda_0^n/\Lambda_1^n=n^{-1}, \Lambda_2^n/\Lambda_1^n=(n+2)/4$ and $\Lambda_0^n/\Lambda_2^n=n^{-1}(\tfrac n 2+1)^{-1}$.
Since $n\geq 4$ for any quotient of the complex hyperbolic space, the above bound implies
\eqref{eq:lower}.
This concludes the proof of the lemma.
\end{proof}
\begin{rmk}
In the previous computation, if we applied the second bound of Lemma \ref{lemm1} $($instead of the first one$)$ to the term $\|Xf\|_{L^2(SM)}^2+\|Hf\|_{L^2(SM)}^2$, we would obtain a lower bound of the form $-O\big(\tfrac 1 {n}\big)\|\nabla_{g_0}S_0\|_{L^2(S^2T^*M)}^2$ in \eqref{eq:lower}. This would give an asymptotically better lower bound $($when $n\to +\infty)$. Nevertheless, this would not make the following Weitzenböck argument work for small values of $n$ so we refrain from using it.
\end{rmk}
\textbf{Using a Weitzenböck formula.} 
We conclude the proof of Proposition \ref{propcoercive}. In the following, $\langle \cdot, \cdot\rangle$ denote the scalar product on $L^2(S^2T^*M)$. We use the fact that $W(S)=Q_1(S)+\tfrac 18Q_0(S)$, equation \eqref{eq:Q_1eq} and the lower bound \eqref{eq:lower}:
\begin{equation}
\label{eq:presquefini}
\begin{split}\langle W(S),S\rangle&\geq \frac 1 4 \left( 1-\frac{97}{192}\right)\langle \nabla^* \nabla S_0,S_0\rangle-\frac 12 \|S_0\|^2+\frac12 \langle S_0\circ J,S_0\rangle 
\\&+\frac 1 4\left(n- \frac{n+5}{8}\right)\|\nabla h\|^2+\frac{n^2}{2}\|h\|^2
\\&\geq  \frac 1 4\times  \frac{95}{192}\langle \nabla^* \nabla S_0,S_0\rangle-\frac 12 \|S_0\|^2+\frac12 \langle S_0\circ J,S_0\rangle +\frac{n^2}{2}\|h\|^2.
\end{split}
\end{equation}
To obtain \eqref{eq:final} from \eqref{eq:presquefini} it remains to obtain a lower bound of the form $C_n\|S_0\|^2$ for some $C_n>0$. 
We use the Weitzenböck formula \eqref{eq:Weit2}, indeed, one has
\begin{align*}
\frac 1 4 \times  \frac{95}{192}&\langle \nabla^* \nabla S_0,S_0\rangle-\frac 12 \|S_0\|^2+\frac12 \langle S_0\circ J,S_0\rangle 
\\ &\geq \frac 1 4\left( \frac{95}{192}((n+3)\|S_0\|^2-3\langle S_0\circ J,S_0\rangle)-2\|S_0\|^2+2\langle S_0\circ J,S_0\rangle \right)
\\&= \frac{1}{4\times 192}\big((95n-99)\|S_0\|^2+99\langle S_0\circ J,S_0\rangle \big).
\end{align*}
We can use Cauchy-Schwartz's inequality and the fact that $J$ is an isometry to bound
$$|\langle S_0\circ J, S_0 \rangle|\leq \|S_0\|\times \|S_0\circ J\|=\|S_0\|^2. $$
Then we get finally that 
\begin{equation}
\label{eq:coerS_0}
\frac 1 4 \times  \frac{95}{192}\langle \nabla^* \nabla S_0,S_0\rangle-\frac 12 \|S_0\|^2+\frac12 \langle S_0\circ J,S_0\rangle\geq \frac{1}{4\times 192}\underbrace{(95n-198)}_{>0}\|S_0\|^2,
\end{equation}
if $n>2$. Plugging \eqref{eq:coerS_0} into \eqref{eq:presquefini} concludes the proof of Proposition \ref{propcoercive} and thus the proof of Proposition \ref{sinj}.
\end{proof}

\printbibliography[
heading=bibintoc,
title={References}]

\end{document}